\documentclass[11pt,a4paper]{article}
\usepackage{amsmath,amsthm}
\usepackage{amssymb,bbm}
\usepackage[utf8]{inputenc}
\usepackage[english]{babel}
\newtheorem{theorem}{Theorem}[section]
\newtheorem{lemma}[theorem]{Lemma}
\newtheorem{corollary}[theorem]{Corollary}
\newtheorem{proposition}[theorem]{Proposition}

\newtheorem{remark}[theorem]{Remark}
\newtheorem{definition}[theorem]{Definition}
\newtheorem{example}[theorem]{Example}

\numberwithin{equation}{section}

\newcommand\Tr{\mathrm{Tr}}
\newcommand\esp{\mathbb E}
\newcommand\var{\mathbb Var}

\newcommand\limN{\underset{N \rightarrow \infty}\longrightarrow}
\newcommand\Nlim{\underset{N \rightarrow \infty}\lim}
\newcommand\eqa{\begin{eqnarray}}
\newcommand\qea{\end{eqnarray}}
\newcommand\eq{\begin{eqnarray*}}
\newcommand\qe{\end{eqnarray*}}

\newcommand\etc{,\ldots ,}
\newcommand\MN{\mathrm{M}_N(\mathbb C) }
\newcommand\DN{\mathrm{D}_N(\mathbb C) }
\newcommand\Dxy{\mcal D^{(\mbf x, \mbf y)}}

\newcommand\one{\mathbbm{1}}
\newcommand\mbf{\mathbf}
\newcommand\mcal{\mathcal}
\newcommand\mbb{\mathbb}

\title{Freeness over the diagonal and global fluctuations of complex Wigner matrices}

\author{Camille Male  \\
\\  UMI Laboratorio Solomon Lefschetz, Mexico \\ \& Centre National de la Recherche Scientifique}

\begin{document}
\date{}
\maketitle
\begin{abstract}We characterize the  limiting 2$^{nd}$ order distributions of certain independent complex Wigner and deterministic matrices using Voiculescu's notions of freeness over the diagonal. If the Wigner matrices are Gaussian, Mingo and Speicher's notion of 2$^{nd}$ order freeness gives a universal rule, in terms of marginal 1$^{st}$ and 2$^{nd}$ order distribution. We adapt and reformulate this notion for operator-valued random variables in a 2$^{nd}$ order probability space. The Wigner matrices are assumed to be permutation invariant with null pseudo variance and the deterministic matrices to satisfy a restrictive property.
%
\end{abstract}

Primary 15B52, 46L54; secondary 46L53, 60F05. Keywords: 
Free Probability, Large Random Matrices, Freeness with Amalgamation, Central Limit Theorem

\tableofcontents

\section*{Introduction}

\begin{definition}\label{Def:Wigner}We call \emph{Wigner matrix} a Hermitian random matrix
$X_N =\frac{1}{\sqrt N}(x_{ij})$, such that the entries $(x_{ij})_{i\leq j}$ are centered and 
independent, diagonal entries are identically distributed real random variables, 
sub-diagonal entries are identically distributed complex random 
variables, the distribution of $x_{ij}$ does not depend on $N$ and $x_{ij}$ has bounded 
moments of all orders ($\esp(|x_{ij}|^k) < \infty$ for all $i, j, k\geq0$). We say that $X_N$ has \emph{null pseudo-variance} if $\esp\big[x_{ij}^2\big]=0$ for $i\neq j$. 

\end{definition}

By Voiculescu's celebrated asymptotic freeness theorem \cite{Voi91}, we know that free probability theory describes the global joint distribution of a large class of random matrices in large dimension. This result states in particular that if $\mbf X_N=(X_j)_{j\in J}$ is a family of independent Gaussian Wigner matrices and $\mbf Y_N=(Y_k)_{k\in K}$ a family of deterministic matrices which converges in $^*$-distribution to variables $\mbf y = (y_k)_{k\in K}$, then the joint family $\mbf X_N \cup \mbf Y_N$ converges to $\mbf x \cup \mbf y$ where $\mbf x= (x_j)_{j\in J}$ is a free semicircular system, free from $\mbf y$. Convergence is with respect to the expectation of the normalized trace, namely  for all $^*$-polynomials $p\in \mbb C\langle \mbf x, \mbf y, \mbf y^*\rangle$
	\eqa\label{IntroPhi}
		\Phi_N(p)  := \esp\Big[ \frac 1 N \Tr \, p(\mbf X_N, \mbf Y_N, \mbf Y_N^*) \Big] \limN \Phi(p).
	\qea

 The convergence toward a semicircular system is very robust, in particular it holds for non Gaussian Wigner matrices \cite{Dyk93} and for many variations of Wigner's model: with bistochastic variance profile \cite{Shl96}, for the band matrix with of intermediate growth \cite{Au18}, the diluted matrix \cite{Mal17}. Remarkably, asymptotic freeness appears for certain non independent matrices, as for independent GUE matrices along with their transpose and their degree \cite{BDJ06,MP16,CDM}.

In this article, we study the joint global fluctuations, that is the collection $\mbf Z_N=(Z_N(p))_p$ of complex random variables
 indexed by $^*$-polynomials,
  	\eqa\label{Intro:ZN}
		Z_N(p):=  \Tr \, p(\mbf X_N, \mbf Y_N, \mbf Y_N^*) -\esp\big[ \Tr \, p(\mbf X_N, \mbf Y_N, \mbf Y_N^*) \big].
	\qea
The limiting fluctuations for linear statistics of a single real Wigner matrix is characterized by Khorunzhy, Khoruzhenko and Pastur in \cite{KKP96}, with the approach of the Stieltjes transform rather than moments. The limit is Gaussian but is not universal: it depends on the fourth moment $\esp[|x_{ij}|^4]$ of the Wigner matrix $X_N = \big( \frac{x_{ij}}{\sqrt N}\big)_{ij}$ . See \cite[Theorem 2.1]{Gui09} for a statement for moments.

When the Wigner matrices are GUE matrices, Mingo and Speicher prove in \cite{MS06} that  $\mbf Z_N$ converges also to a Gaussian process whenever the family $\mbf Y_N$ have a limiting $^*$-distribution. They define the notion of \emph{second-order freeness} which describes more generally the limit of unitarily invariant matrices at first and second order: in particular, this notion gives a formula for the limiting covariance
	\eqa\label{IntroPhi2}
		\Phi_N^{(2)} (p,q^*) &= & \esp\big[Z_N(p)\overline{Z_N(q)}\big] \limN \Phi^{(2)} (p,q^*)
	\qea
that can be interpreted in terms of the so-called \emph{spoke diagrams} \cite[Definition 6.3 and Figure 11]{MS06}. In general Wigner matrices are not asymptotically 2$^{nd}$ order free: the 2$^{nd}$ order limit of real orthogonal invariant random matrices is given by Redelmeier's notion of \emph{real} 2$^{nd}$ order freeness \cite{Red12}. Covariance computations for these ensembles involve also the so-called \emph{twisted} spoke diagrams. 

In this article, we consider complex Wigner matrices with null pseudo-variance. The dependence in the fourth moment for the fluctuation of Wigner matrices is an obstacle for a universal presentation. Nevertheless, under the hypotheses on the Wigner and deterministic matrices stated in Section 2, we see in this article that Voiculescu's notion of \emph{operator-valued probability space over the diagonal} allows to bypass this difficulty. In a sense properly clarified in next sections, freeness with amalgamation over the diagonal rules the asymptotic fluctuations of theses ensembles thanks to a modification of Mingo and Speicher theory.

A motivation for this article is the following open question in random matrix theory: how to characterize the limit of spectral linear statistics for the sum $X_N+Y_N$ of a Wigner and an independent Hermitian matrix ? Khorunzhy consider in \cite{Kho94} the case of a GOE and a deterministic matrix; Ji and Lee in \cite{JL20} consider real Wigner and diagonal matrices; Dallaporta and F\'evrier in  \cite{DF19} consider complex Wigner (with arbitrary real pseudovariance) and deterministic diagonal matrices. The question remains open when $Y_N$ is not diagonal, for which the global fluctuations depend not only on the spectrum of $Y_N$. Heuristically, studying the asymptotic of the process $\mbf Z_N$ of \eqref{Intro:ZN} we search for a good theoretical framework for this spectral linear statistics problem, with the perspective of finding a good analogue of the \emph{2$^{nd}$ order Stieltjes transform}. Our investigation suggests that \emph{freeness over the diagonal} could be useful for Wigner matrices with zero pseudo-variance, but only with deterministic matrices that satisfy an assumption somehow \emph{opposed to diagonality}, see assumption (H5) and Remark \ref{Rk:Deloc}. Finding a unifying point of view to include the models of \cite{JL20,DF19} remains an open question.

We pursue here a combinatorial analysis initiated in \cite{MMPS}.  
 Under mild hypotheses on the deterministic matrices and assuming the Wigner matrices permutation invariant (see Section \ref{Sec:Statements}), 
 we know from  \cite{MMPS} that $\mbf Z_N$ is bounded. Its possible limits are Gaussian and we have a combinatorial description of the limit of the 2$^{nd}$ order $^*$-distribution. Two important cases can be distinguished.
\begin{enumerate}
	\item Let $\Delta(A)=\mathrm{diag}_i(A_{ii})$ be the diagonal of a matrix $A$. If all Wigner matrices have null pseudo-variance, then the possible limit of $\Phi_N^{(2)}$ depends on the limit of $\esp\big[\frac 1N \Tr  \big( \Delta [p(\mbf Y_N)] \Delta [q(\mbf Y_N)]\big) \big]$, for all polynomials $p,q$.
	\item Otherwise it also depends on the limit of $\esp\big[\frac 1N \Tr \big(P(\mbf Y_N)Q(\mbf Y_N)^t)\big]$ where $A^t$ is the transpose of the matrix $A$. 
\end{enumerate}
Moreover we know that random and deterministic matrices are asymptotically free over the diagonal under mild conditions \cite{ACDGM}. With these techniques we prove that this behavior is somehow \emph{robust} at 2$^{nd}$ order for the matrices considered.

Finally, we observe a common phenomenon in 2$^{nd}$ order freeness and its extension, formulated briefly as follow: \emph{denoting $\Phi:E \to \mbb C$ and $\Phi^{(2)}:E^2 \to \mbb C$ the 1$^{st}$ and 2$^{nd}$-order distributions of a limiting multi-matrix model $\mbf a = \sqcup_{\ell\in L} \mbf a_\ell$, there are canonical subspaces $F_n \subset E$, $n\geq 2$, orthogonal w.r.t. $ \Phi^{(2)}$, where \emph{$\Phi^{(2)}$ is collinear to the canonical bilinear form associated to $\Phi$}, namely}
	\eqa\label{Eq:CollIntro}
		\Phi^{(2)} (p,q^*) = n \Phi (pq^*), \ \ \forall p,q \in F_n;
	\qea
	\emph{the spaces $F_n$ depends only on the class of the model (e.g. unitarily invariant matrices or Wigner with null pseudo-variances) and are associated to a notion of independence (e.g. freeness or freeness over the diagonal).}

 For the matrix model studied in this article, $\Phi^{(2)} $ is not universal so it cannot be determined by $\Phi$ only. Yet we can state a collinearity property \eqref{Eq:CollIntro} to describe partially $\Phi^{(2)} $ by considering the problem on a space $E$ that extends the space of polynomials (with the \emph{$\Delta$-polynomials}). The description of the universal part of $\Phi^{(2)}$ comes with an additional condition, in the form of a Leibniz Formula.
 
\section{Preliminaries}\label{Sec:Prelim}

In order to describe the global fluctuations of random matrices  we propose to use the following abstract setting.
\begin{definition}\label{Def:SecOrdSpace} An \emph{ operator-valued 2$^{nd}$ order probability space}  is the data of a quintuple $(\mcal A, \Phi, \Phi^{(2)}, \mcal D, \Delta)$ as follow.
\begin{enumerate}
	\item $(\mcal A,\Phi)$ is a \emph{tracial $^*$-probability space}, namely $\mcal A$ is a $^*$-algebra and $\Phi$ is a unital and tracial linear form on $\mcal A$: $ \Phi[1] = 1, \ \ \Phi[ab]=\Phi[ba] \ \forall a,b \in \mcal A.$
	\item $(\mcal A, \Phi, \Phi^{(2)})$ is a \emph{2$^{nd}$ order probability space}, i.e $\Phi^{(2)}$ is a symmetric bilinear form on $\mcal A$, tracial in each variable, s.t. $\Phi^{(2)}[1,a] = \Phi^{(2)}[a,1]=0 \forall a\in \mcal A$.
	\item $(\mcal A,  \mcal D, \Delta)$ is an \emph{operator-valued probability space}, namely $\mcal D \subset \mcal A$ a unital subalgebra and  $\Delta: \mcal A \to \mcal D$ is a conditional expectation, i.e. a unital linear map s.t. $
			\Delta(d_1ad_2) = d_1 \Delta(a) d_2, \ \forall a\in \mcal A, \forall d_1,d_2 \in \mcal D.$
	\item We furthermore assume that $\Phi$ is \emph{invariant under $\Delta$}, i.e $ \Phi\big[ \Delta(a)\big] = \Phi[a]$ for any $ a\in \mcal A$, and that $\Phi^{(2)}$ is invariant under $\Delta$ in both variables. 
\end{enumerate}
\end{definition}

Let $\MN$ be the set of deterministic $N$ by $N$ complex matrices and $\DN$ the set of diagonal matrices. For any $N\geq 1$, fix  $\Omega_N$ a probability space in the classical sense. Denote  by $\mathrm L^{\infty-}\big( \MN \big)$ and $\mathrm L^{\infty-}\big(\DN \big)$ the spaces of random and random diagonal matrices on $\Omega_N$ respectively, whose entries have finite moments of all orders. We set $\Phi_N:A\mapsto \esp\big[ \frac 1 N \sum_{i=1}^N A_{ii}\big]$ the expectation of the normalized trace on $\mathrm L^{\infty-}\big( \MN \big)$, we set  the covariance of traces function
	$$\Phi_N^{(2)}:(A,B) \mapsto  \esp\Big[ \big( \Tr \,  A - \esp[ \Tr \,  A] \big) \big( \Tr  \, B - \esp[ \Tr  \, B] \big)\Big],$$
and we set $\Delta:A\mapsto  \mathrm{diag}_i(A_{ii})$ the projection of a random matrix into its diagonal part. We then get an operator-valued 2$^{nd}$ order probability space  with random coefficients.
	$$\big( \mathrm L^{\infty-}\big(  \MN \big),\Phi_N, \Phi^{(2)}_N ,\mathrm L^{\infty-}\big( \DN \big),  \Delta \big).$$

Let $\mbf X_N$ be a collection of independent Wigner matrices and $\mbf Y_N$ be deterministic matrices. As presented previously, our purpose is to understand the limit $(\mbf x, \mbf y)$ of $(\mbf X_N, \mbf Y_N) $ as $N$ goes to infinity in the sense of 2$^{nd}$ order probability spaces and compute $\Phi^{(2)}[p,q]= \lim \Phi_N^{(2)}\big[ p(\mbf X_N, \mbf Y_N), q(\mbf X_N, \mbf Y_N)\big]$ for $p,q$ $^*$-polynomials. The convergence is already proved in \cite{MMPS}, but here the additional operator-valued setting provides a more natural presentation of the limit.

We do not construct a 2$^{nd}$ order theory of operator-valued probability, with a bilinear map  $\Delta^{(2)}$ describing \emph{fluctuation operators}, e.g. $N(\Delta(A) - \esp[A])$ for a random matrix $A$. Also, contrary to the other applications of freeness over the diagonal in \cite{Shl96,BD17} we need \emph{random} diagonal matrix coefficients. The strategy in this article is to introduce only the \emph{ad-hoc} material sufficient to solve our matrix problem, based on a restriction of the notion of traffics distribution \cite{Mal11}, for which we already know all asymptotics of interest.

Let $J$ be an index set and let $\mbf x=(x_j)_{j\in J}$ be non commutative indeterminates. We set $\mbb C\langle  \mbf x   \rangle$ the space of non commutative polynomials in $\mbf x$.  We recall that a \emph{$\mcal D$-valued monomial in the variables $\mbf x$} is a word $p = d_0 x_{j_1}  d_1 \cdots  x_{j_n} d_n,$ where $n\geq 0$, $d_0 \etc d_L \in \mcal D$ (called the \emph{coefficients} of $p$), and $j_1 \etc j_n  \in J$. We set $\mcal D\langle \mbf x  \rangle$ the space of $\mcal D$-valued polynomials, finite sums of $\mcal D$-valued monomials in $\mbf x$. 

\begin{definition}\label{Def:OpValDistr}Let $(\mcal A, \mcal D, \Delta)$ be an operator-valued probability space.
\begin{itemize}
	\item  The \emph{operator-valued distribution} $\Delta_{\mbf a}$ of a collection $\mbf a=(a_j)_{j\in J} \in \mcal A^J$  is the map $\Delta_{\mbf a} : p\mapsto \Delta\big(p(\mbf a) \big)\in \mcal D$, defined for all $ p\in  \mcal D\langle  \mbf x \rangle$. 
	\item A tuple $(a_1 \etc a_n)$ is \emph{$\mcal D$-alternating} in the collections $\mbf a_\ell$, $\ell \in L$, if $n\geq 2$ and $a_i= p_i(\mbf a_{\ell_i}) $ for indices $\ell_1 \neq \ell_2 \neq \cdots \neq \ell_n \in L$ where $p_i$ is a $ \mcal D$-valued polynomial. The tuple is \emph{$\Delta$-centered} if  $\Delta( a_i )  = 0, \forall i=1\etc n$. 
	\item The collection $\mbf a_\ell, \ell \in L$, are \emph{free over $\mcal D$} whenever $\Delta( a  ) = 0$ for any $a=a_1\cdots a_n$ where $n\geq 2$ and $(a_1 \etc a_n)$  is $\mcal D$-alternating and $\Delta$-centered.
\end{itemize}
\end{definition}

We instead use the following notion of \emph{$\Delta$-distribution}, based on a representation of operator-valued polynomials coefficients. This approach is used in \cite{ACDGM} in the case of matrices, we introduce it in the general case and in a simpler way. 
 
\begin{definition}\label{Def:DeltaPolynome} Let $\mbf x = (x_j)_{j\in J}$ be a collection of indeterminates. 
	\begin{enumerate}
	\item A \emph{bracketed word} in $\mbf x$ is a word in the $x_j$ and in a left and a right bracket symbols $[$ and $]$, that belongs to the smallest monoid by concatenation containing $\mbf x$ and stable by the \emph{bracketing operation} $\Delta : w \mapsto [w]$. 
	\item The \emph{set of $\Delta$-monomials in the variables $\mbf x$} is the quotient monoid given by the relations: $\Delta(1)  = 1$, where $1$ is the empty word, and $\forall \omega_1,\omega_2,\omega_3$
		\eq	
			\Delta\big( \Delta( \omega_1)\omega_2\Delta( \omega_3)  \big) & \sim & \Delta( \omega_1)\Delta(\omega_2)\Delta( \omega_3).		\qe
\end{enumerate}
\end{definition}

	We set  $\mathbb C\langle  \mbf x \rangle_\Delta$ the space of \emph{$\Delta$-polynomials}, i.e. finite complex linear combinations of $\Delta$-monomials in $\mbf x$.  Product and evaluation of $\Delta$ are extended linearly. For any $q \in \mathbb C\langle  \mbf x \rangle_\Delta$ and for any collection $\mbf a = (a_j)_{j\in J}$ of elements of an operator-valued probability space $(\mcal A, \mcal D, \Delta)$, we set $q(\mbf a)\in \mcal A$ by substituting indeterminates by elements of $\mbf a$ and interpreting brackets as the evaluation of $\Delta$. 

\begin{definition}\label{Def:OpValNotions} Let $(\mcal A, \Phi, \Phi^{(2)}, \mcal D, \Delta)$ be an operator-valued 2$^{nd}$ order probability space and $\mbf a=(a_j)_{j\in J} \in \mcal A^J$. For each $N\geq 1$, let $\mbf A_N=(A_{N,j})_{j\in J}$ be a collection of random matrices in $\mathrm L^{\infty-}\big(  \MN \big)$. Let $\mbf x=(x_j)_{j\in J}$ be indeterminates. 		\begin{enumerate}
		\item The \emph{$\Delta$-distribution} of $\mbf a$  is the map	$\Phi_{\mbf a}: p \mapsto \Phi\big[p(\mbf a)\big]$ defined for $ p \in \mbb C\langle  \mbf x , \mbf x^* \rangle_\Delta$. We say that $\mbf A_N$ \emph{converges to $\mbf a$ in $\Delta$-distribution} whenever the $\Delta$-distribution converges pointwise: $\forall p \in \mbb C\langle  \mbf x , \mbf x^* \rangle_\Delta$, $ \Phi_{\mbf A_N}\big[ p \big] \limN \Phi_{\mbf a}\big[ p\big]$.
		\item The \emph{2$^{nd}$ order $\Delta$-distribution} of $\mbf a$ is the map $\Phi^{(2)}_{\mbf a}: (p,q) \mapsto \Phi^{(2)}\big[p(\mbf a), q(\mbf a)\big]$, defined for $p,q \in \mbb C\langle  \mbf x , \mbf x^* \rangle_\Delta$. The sequence $\mbf A_N$ \emph{converges to $\mbf a$ in 2$^{nd}$ order $\Delta$-distribution} whenever it converges to $\mbf a$ in $\Delta$-distribution and the process of random variables  $\mbf Z_N =  ( \Tr \, p(\mbf A_N) - \esp[ \Tr \, p(\mbf A_N)]  )_{p\in \mbb C\langle  \mbf x , \mbf x^* \rangle_\Delta}$
converges in law to the centered Gaussian process $\mbf z=(z_p)_{p\in \mbb C\langle  \mbf x , \mbf x^* \rangle_\Delta}$ with covariance function $\esp[ z_{p_1}\overline{z_{p_2}}]  =  \Phi^{(2)}_{\mbf a}[ p_1,p_2^*]$.
	\end{enumerate}
\end{definition}

We call \emph{$\Delta$-algebra generated by $\mbf a$} the set $ \mathbb C\langle  \mbf a \rangle_\Delta:=\{p(\mbf a), \forall p \in \mathbb C\langle  \mbf x \rangle_\Delta\}$. It is the smallest unital subalgebra of $\mcal A$ closed by $\Delta$ that contains $\mbf a$. We say that $a\in  \mathbb C\langle  \mbf a \rangle_\Delta$ is $\Delta$-invariant if $\Delta(a)=a$, and denote by $\mcal D^{(\mbf a)}$ the set of $\Delta$-invariant elements of $ \mathbb C\langle  \mbf a \rangle_\Delta$. The triplet $(\mathbb C\langle  \mbf a \rangle_\Delta, \mcal D^{(\mbf a)}, \Delta)$ is an operator-valued probability space, and the $\Delta$-distributions of $\mbf a$ is given by the restriction of $\Phi$ on $ \mathbb C\langle  \mbf a \rangle_\Delta$.

Note that there is no known analogue of freeness for the $\Delta$-distributions. In particular $\mcal D$-freeness does not determine entirely the $\Delta$-distribution, see remark \ref{commentAn} and Lemma \ref{WhyA1}. This is not important in this article since we consider a particular situation: at 1$^{st}$ order the algebra of coefficient $\mcal D^{(\mbf a)}$ is isomorphic to $\mbb C$, see Hypothesis (H5) next section. This implies that the  $\Delta$-distribution follows trivially from the $^*$-distribution. 

\section{Statements of the results}\label{Sec:Statements}

We now state hypotheses on matrices and non commutative random variables that we use in the statements below. Let  $X_N=(\frac{x_{ij}}{\sqrt N})_{i,j}$ be a Wigner matrix.
	\begin{enumerate}
		\item [{(H1)}] $X_N$ is invariant in law by the permutation group, or equivalently the entry $x_{12}$ is distributed as its complex conjugate $\bar x_{12}$.
		\item [{(H2)}]  The \emph{pseudo-variance} $\esp(x_{ij}^2)$ of each nondiagonal entry $i\neq j$ equals zero.	
	\end{enumerate}
Let $\mbf Y_N$ be a collection of deterministic matrices.
	\begin{enumerate}
		\item [{(H3)}] $\sup_{N} \| Y_N\| < \infty$ for any $Y_N$ matrix of $\mbf Y_N$, for $\| \, \cdot\, \|$ the operator norm.
		\item [{(H4)}] $\mbf Y_N$ has a limiting $\Delta$-distribution: for any $p \in \mathbb C \langle \mbf y , \mbf y^*\rangle_\Delta$, the limit exists
	 $$\frac 1 N \Tr \, p(\mbf Y_N, \mbf Y_N^*) \limN \Phi [ p  ],$$
	where $\Delta$ acts on $\MN$ by projecting a matrix on its diagonal part. 
	\end{enumerate}
Let $\mbf a$ be a collection of elements of $(\mcal A, \Phi, \Phi^{(2)}, \mcal D, \Delta)$ as in Definition \ref{Def:OpValNotions}. We will invoke this property which means that $\Delta$ is trivial at 1$^{st}$ order (see Remark \ref{Rk:Deloc}):	\begin{enumerate}
		\item  [{(H5)}] Any $\Delta$-invariant element $a\in \mcal D^{(\mbf a)}$ of the $\Delta$-algebra generated by $\mbf a$ is distributed w.r.t. $\Phi$ as the deterministic scalar $\Phi(a)1$: $\forall a_1,a_2 \in  \mbb C\langle \mbf a \rangle_\Delta$ 
	\eqa\label{Eq:TransitivityShort}
		\Phi\big( \Delta(a_1) \Delta(a_2)\big) = \Phi (a_1) \times \Phi(a_2).
\qea
\end{enumerate}

Condition (H1) is inherent to our method and presumably technical; (H2) and (H5) are crucial for our statements. Without loss of generality, we will assume $\mbf Y_N$ closed by adjoint: if a matrix $Y_N$ belongs to $\mbf Y_N$, then so does the adjoint $Y_N^*$. This simplifies notations, e.g. with $\mbb C\langle \mbf x, \mbf y\rangle$ instead of $\mbb C\langle \mbf x, \mbf y,  \mbf y^*\rangle$.

We first study the 1$^{st}$ order $\Delta$-distribution and state a slight generalisation of the asymptotic freeness of Wigner and deterministic matrices.

\begin{theorem}\label{MainTh1} Let $\mbf  X_N$ be a collection of independent Wigner matrices satisfying (H1) and $\mbf  Y_N$ be a collection of deterministic matrices satisfying (H3), (H4). Then $(\mbf X_N, \mbf Y_N)$ converges in $\Delta$-distribution to  a collection of variables $(\mbf x, \mbf y)$. The convergence
\eqa\label{CVMainTh1}
	\frac 1N \Tr \ p(\mbf X_N, \mbf Y_N) \limN \Phi\big[ p(\mbf x, \mbf y) \big],\  \forall p \in \mbb C\langle \mbf x, \mbf y \rangle_\Delta
\qea
holds in expectation, in probability and almost surely. Furthermore for any $\Delta$-monomial $m$ in the variables $(\mbf x, \mbf y)$ and for any variable $x$ in the collection $\mathbf x$, we have
	\eqa\label{SDeq}
		\Phi [ x m  ] =\sum_{ m = \ell x r}\Phi[x^2] \Phi [\ell] \Phi [r],
	\qea
where the sum is over all decompositions of $m$ as a product of $\Delta$-monomials $\ell \times x \times r$. Finally, $\mbf x$ satisfies (H5) and if the limit $\mbf y$ of $\mbf Y_N$ also satisfies (H5) then so does $(\mbf x, \mbf y)$.
\end{theorem}
The theorem is proved in Section \ref{Sec:FirstOrder}. The convergence \eqref{CVMainTh1} and Formula \eqref{SDeq} are known  for ordinary $^*$-polynomials instead of $\Delta$-polynomials \cite[Lemma 5.4.7]{AGZ10}. In this case Formula \eqref{SDeq} is called the Schwinger-Dyson equation and it is a classical consequence of freeness, see \cite[Theorem 14.4]{NS06}.  
In Theorem \ref{MainTh1} this equation characterizes the $\Delta$-distribution of $(\mbf x, \mbf y)$ in terms of the $\Delta$-distribution of $\mbf y$, as proved in Section \ref{Sec:NCprop} and illustrated in the example below.

\begin{example}\label{example1GBC} Assuming $\Phi(x^2)=1$, we compute $\alpha:=\Phi [ x^2\Delta(x^2y)yxyx ]$. Applying \eqref{SDeq} with $m= x\Delta(x^2y)yxyx $, the different decompositions give
 	$$\Phi [ x^2\Delta(x^2y)yxyx ]  =   \Phi[\Delta(x^2y)yxyx] + \Phi\big( x  \Delta(x^2y)y\big)\Phi(yx)+ \Phi( x\Delta(x^2y)yx y)  .$$
The middle term $\Phi\big( x  \Delta(x^2y)y\big)$ vanishes by \eqref{SDeq} since $q=\Delta(x^2y)y$ has no decomposition $q=\ell x r$. Moreover, by traciality the first and third terms are equal. By \eqref{SDeq} again we have $ \Phi[x\Delta(x^2y)yxy] = \Phi\big[\Delta(x^2y)y\big] \Phi[y]$. To compute $\Phi\big[\Delta(x^2y)y\big]$, by $\Delta$-invariance of $\Phi$ and the property of conditional expectation, we write
	\eqa
		 \Phi\big[\Delta(x^2y)y\big]  &= & \Phi\Big[ \Delta \big[ \Delta(x^2y)y \big] \Big] =  \Phi\Big[ \Delta(x^2y)\Delta  [ y  ] \Big]\nonumber\\
		 &  = & \Phi\Big[ \Delta\big[ x^2y \Delta  [ y  ]\big] \Big] = \Phi\big[x^2y\Delta(y)\big]. \label{TrickBC}
	\qea
By \eqref{SDeq} again, we get $\Phi\big[\Delta(x^2y)y\big] =  \Phi\big[\Delta(y)^2\big]$ and finally $\alpha = 2\Phi\big[\Delta(y)^2\big] \Phi[y]$.
\end{example}

\begin{remark}\label{Rk:Deloc}The last part in Theorem \ref{MainTh1} says, in terms of matrices, that if 
\eqa\label{RefH5} \frac 1 N \Tr \, \big[ \Delta\big(A_N \big) \Delta\big(A_N\big)^* \big]  -  \Big | \frac 1 N \Tr \, \big[  A_N \big] \Big|^2 \limN 0\qea  
for any $A_N = p(\mbf Y_N), p\in \mbb C\langle \mbf y \rangle_\Delta$, then \eqref{RefH5}  holds for any $A_N = p(\mbf X_N, \mbf Y_N), p\in \mbb C\langle \mbf x , \mbf y \rangle_\Delta$. In particular:
\begin{itemize}
	\item If $A_N$ is a diagonal matrix that is not asymptotically distributed as a scalar matrix, then \eqref{RefH5} does not hold. 
	\item If $A_N$ is a Hermitian matrix whose eigenvectors $u$ are \emph{delocalized}, i.e. satisfy a uniform estimate $|\langle u, e_j\rangle| =O(N^{-1/2})$ for the scalar product with elementary basis vectors $e_j$, then \eqref{RefH5} holds  (by an easy computation writing $A_N = \sum_{i=1}^N \lambda_i u_i u_i^*$ in eigenvectors basis).
\end{itemize}
\end{remark}

\begin{example} Assumption (H5) holds almost surely for realizations of unitarily invariant matrices satisfying (H4) \cite[Theorem 1.1]{CDM} and of uniform permutations matrices \cite[Section 3.2.2]{Mal11}. Traffics algebras satisfying (H5) are \emph{those for which traffic independence implies free independence} \cite[Corollary 2.9]{CDM}. 
\end{example}

We now consider our initial problem on 2$^{nd}$ order distributions. Next definition is fundamental for the study of $\Delta$-distribution, see Section \ref{Sec:NCprop}.

\begin{definition}\label{def:An} In a 2$^{nd}$ order operator-valued prob. space $(\mcal A, \Phi, \Phi^{(2)}, \mcal D, \Delta)$, we fix $\mbf a_\ell$, $\ell \in L$ collections of elements of $\mcal A$, and let $\mbf a$ denote the union of all $\mbf a_\ell$'s.
\begin{enumerate}
	\item For any $n\geq 2$, a tuple $(a_1 \etc a_n)$ is \emph{cyclically $\mcal D^{(\mbf a)}$-alternating}  if $a_i= p_i(\mbf a_{\ell_i}) $ for indices $\ell_1 \neq \cdots \neq \ell_n \neq \ell_1$ where $p_i$ is a $ \mcal D^{(\mbf a)}$-valued polynomial. We set $E_n$ the vector space generated by all products $a_1\cdots a_n$ where $(a_1\etc a_n)$ is cyclically $\mcal D^{(\mbf a)}$-alternating and $\Delta$-centered.
	\item  	We set $E_1$ the smallest vector space s.t. $\mcal Alg(\mbf a_\ell, \Delta( E_1) ) \subset E_1, \forall \ell \in L$. \end{enumerate}
\end{definition}

\begin{remark}\label{commentAn} If the $\mbf a_\ell$'s are free over $\mcal D$ (in the ordinary sense) then $\Phi$ vanishes on $E_n$ for $n\geq2$; but there is no a priori knowledge about $\Phi$ on $E_1$. 
\end{remark}

\begin{example} For each $i=1,2,\dots$ let $\ell_i\in L$ and $a_{\ell_i} \in \mbf a_{\ell_i}$.
\begin{itemize}
	\item $m=\big[a_{\ell_1}\Delta(a_{\ell_2}a_{\ell_3})a_{\ell_1}-\Delta\big(a_{\ell_1}\Delta(a_{\ell_2}a_{\ell_3})a_{\ell_1}\big)\big]  \big(a_{\ell_3}-\Delta(a_{\ell_3})\big) \in E_2$ if $\ell_1\neq \ell_3$. 
	\item $m =  \Delta(a_{\ell_1})a_{\ell_2} \Delta\big( a_{\ell_3}\Delta(a_{\ell_4}^2) a_{\ell_3}\big) a_{\ell_2}$ is in $E_1$. Indeed,  $m \in \mcal Alg(\mbf a_2, \Delta( E_1) )$ since we have $m = d_1a_{\ell_2} d_2 a_{\ell_2}$ where $d_1 \in \mcal D^{(\mbf a_1)} $ and $d_2 \in \Delta\big(\mcal Alg( \mbf a_3, \mcal D^{(\mbf a_4)})  \big)$. 
\end{itemize}
\end{example}

The analysis of Wigner matrices motivates the following definition.

\begin{definition}\label{LaMalNommeeDelta2Freeness} Let $(\mcal A, \Phi, \Phi^{(2)}, \mcal D, \Delta)$, $\mbf a_\ell,\ell \in L,$ and $E_n, n\geq 1$, be as in the previous definition, and assume that the collection of all $\mbf a_\ell$'s satisfy (H5). The families $\mbf a_\ell$ are \emph{2$^{nd}$ order $\Delta$-free} if they are free w.r.t. $\Phi$ and the following holds. 
	\begin{enumerate}
		\item \emph{Orthogonality conditions:} the vector spaces $E_n, n\geq 1,$ are orthogonal w.r.t. $\Phi^{(2)}$, i.e. for any $a\in E_{n}, b\in E_{m}$, $n\neq m$, then $\Phi^{(2)}(a,b^*)=0$; moreover, the $^*$-algebras generated by $\mbf a_\ell,\ell \in L$, are orthogonal w.r.t. $\Phi^{(2)}$, i.e. for any $^*$-polynomials $p$ and $ q$, one has $\Phi^{(2)}\big( p(\mbf a_\ell), q(\mbf a_{\ell'})^*)=0$ if $\ell \neq \ell'$.
		 \item \emph{Mingo-Speicher formula, $n\geq 2$:} for all $ (a_1\etc a_n), (b_1\etc b_n)$ cyclically $\mcal D^{(\mbf a)}$-alternating and  $\Delta$-centered, with $a=a_1\cdots a_n$, $b=b_1\cdots b_n$ we have
	\eqa\label{Eq:Def2ndOrderFreeMS}
		\Phi^{(2)}(a, b^*) =   \sum_{i=0}^{n-1}  \Phi(a_1b^*_{i+1}) \Phi(a_2b^*_{i+2}) \cdots \Phi(a_nb^*_{i+n}),
		\qea
where indices of $b_k$'s are counted modulo $n$ (e.g. $b_{n+1}=b_1$).
		\item \emph{Leibniz rule, $n=1$:} for all $a,b,c \in E_1$, we have
	\eqa\label{LeibnizRule}
			 \Phi^{(2)}\big( \Delta(a) \Delta(b) , c \big) =  \Phi^{(2)}\big( a , c \big) \Phi(b) +  \Phi(a)\Phi^{(2)}\big( b , c\big).
	\qea

	\end{enumerate}
\end{definition}

\begin{remark} We recall that no general notion of 1$^{st}$ order $\Delta$-freeness exists  so naming \emph{2$^{nd}$ order $\Delta$-freeness} is an abuse of terminology. Moreover, we name item (ii) \emph{Mingo-Speicher formula} since it is formally identical to the spoke diagram formula of ordinary 2$^{nd}$ order freeness, but it is stated in other spaces. 
\end{remark}

In Section \ref{Sec:NCprop} we prove that the rules (i), (ii) and (iii) determine uniquely the 2$^{nd}$ order $\Delta$ distribution of $\mbf a$ in terms of the marginal 1$^{st}$ and 2$^{nd}$ order $^*$-distributions of each $\mbf a_\ell$. We illustrate the method below.

\begin{example} We compute $\Phi^{(2)}(m, m^*)$ for $m=a_1 a_2\Delta(a_1a_2)$ where $a_i\in \mbf a_{\ell_i}, \ell_1\neq \ell_2$. We have $\Phi^{(2)}[m, \cdot] = \Phi^{(2)}[m_1 + m_1', \cdot]$ where $m_1 =  \big(a_1-\Delta(a_1)\big)  \big(a_2-\Delta(a_2)\big) \Delta(a_1a_2)  $ is in the space $E_2$ and $m_1'= - \Delta(a_1)  \Delta(a_2)  \Delta(a_1a_2)$. Moreover, using the same computation as for \eqref{TrickBC}, we have 
$\Phi^{(2)}(m'_1, \cdot)= \Phi^{(2)}(m''_1, \cdot)$ where $m''_1 = \Delta(a_1)  \Delta(a_2)  a_1a_2$. We write as before $\Phi^{(2)}[m''_1, \cdot]  =  \Phi^{(2)}[m_2+m_3, \cdot]$  where $m_2=\Delta(a_1)  \Delta(a_2) \big(a_1-\Delta(a_1)\big)  \big(a_2-\Delta(a_2)\big) \in E_2$ and $m_3 = \Delta(a_1) \Delta(a_2)\Delta(a_1) \Delta(a_2)$ is in $E_1$. Hence by the first orthogonality rule $\Phi^{(2)}(m,m^*) = \alpha_1 + \alpha_2 + \alpha_3$ where $\alpha_i =  \Phi^{(2)}(m_i,m_i^*)$. By Mingo-Speicher formula we have $\alpha_1 =   \Phi\big[ \big(a_1-\Delta(a_1)\big)\big(a_1-\Delta(a_1)\big)^* \big] \times    \Phi\big[ \big(a_2-\Delta(a_2)\big)\big(a_2-\Delta(a_2)\big)^*\big]  \big| \Phi(a_1 a_2)\big|^2.$ Moreover, by freeness of $a_1$ and $a_2$ and by (H5) we find $\alpha_1 = \var( a_1)\var( a_2)  \big| \Phi(a_1) \Phi( a_2)\big|^2$, where $\var(a) =  \Phi\big[ \big(a-\Phi(a)\big)\big(a-\Phi(a)\big)^* \big]$.  Similarly one shows $\alpha_1=\alpha_2$. By the Leibniz rule with $a =  \Delta(a_1)^2$ and $b=  \Delta(a_2)^2 $ and the second orthogonality rule, we have 
		$\alpha_3 =
		  \Phi^{(2)} \big( \Delta(a_1)^2  , \Delta(a_1^*)^2  \big) \big|\Phi\big(  \Delta(a_2)^2 \big)\big|^2 +  \big|\Phi\big(  \Delta(a_1)^2 \big)\big|^2 \Phi^{(2)} \big( \Delta(a_2)^2  , \Delta(a_2^*)^2  \big)$. Using the Leibniz rule with $a =b=a_i$, we get
		$\alpha_3 =  4|\Phi(a_1)|^2 \Phi^{(2)}(a_1,a_1^*)|\Phi(a_2)|^4 +4|\Phi(a_1)|^4 |\Phi(a_2)|^2 \Phi^{(2)}(a_2,a_2^*) $. Note that the $\alpha_i$'s have been written in terms of 1$^{st}$ and 2$^{nd}$ order ordinary moments of $a_1$ and of $a_2$. 
\end{example}

\begin{theorem}\label{MainTh2} Let $\mbf  X_N$ and $\mbf  Y_N$ be collections of independent Wigner and deterministic matrices satisfying (H1), (H2), (H3), (H4) and such that the limit $\mbf y$ of $\mbf Y_N$ satisfies (H5). Then $(\mbf X_N, \mbf Y_N)$ converges in 2$^{nd}$ order $\Delta$-distribution to a collection of variables $(\mbf x, \mbf y)$ such that the variables of $\mbf x$ are \emph{2$^{nd}$ order $\Delta$-free}, and are \emph{2$^{nd}$ order $\Delta$-free} from $\mbf y$. 
\end{theorem}

The theorem is proved in Section \ref{Sec:SecondOrder}. We now reformulate property (ii).

\begin{proposition}\label{Collin} Let $(\mcal A, \Phi, \Phi^{(2)}, \mcal D, \Delta)$, $\mbf a_\ell,\ell \in L,$ and $E_n, n\geq 2$, be as in Definition \ref{def:An}. Assume that the families $\mbf a_\ell$'s are free w.r.t. $\Phi$ and denote
	$$F_n = \textrm{Span} \, \Big\{  \sum_{i=0}^{n-1} a_{i+1} \cdots a_{i+n}, \, \Big| \, (a_1\etc a_n) \ \Delta \textrm{-centered \ cyclically \ }\mcal D^{(\mbf a)}\textrm{-alternating} \ \Big\},$$
where indices are counted modulo $n$. Then Mingo-Speicher formula (ii) is valid for all $n\geq 2$ if and only if $ \Phi^{(2)}(p,q^*) = n \Phi(pq^*)$, for all $n\geq 2$ and all $p,q\in F_n$.
\end{proposition}

Hence, back to the central limit, by \emph{plugging-in} the 1$^{st}$ order estimators we get for a large matrix $M_N=P(\mbf X_N, \mbf Y_N)$ in  $F_n $ a Gaussian approximation
	\eqa
		 \frac 1 N \Tr M_N \sim  \mcal N \left( \esp\Big[  \frac 1 N \Tr M_N  \Big], \frac{  \esp\big[\frac {n} N \Tr (M_NM_N^*)\big]}{ N^2}\right).
	\qea

\begin{proof}[Proof of Proposition \ref{Collin}] For any $\Delta$-centered  cyclically $\mcal D^{(\mbf a)}$-alternating tuples $(a_1\etc a_n)$ and $(b_1\etc b_n)$, let $a=a_1\cdots a_n$, $\hat a = \frac 1 n \sum_{i=0}^{n-1}a_{1+i} \cdots a_{n+i}$ and similarly  let $b$ and $\hat b$. By traciality, $\Phi^{(2)}(a,b^*) = \Phi^{(2)}(\hat a,\hat b^*)$, so (ii) is valid whenever $\Phi^{(2)}(\hat a,\hat b^*) = \sum_{i=0}^{n-1} \Phi(a_1b^*_{i+1}) \cdots \Phi(a_nb^*_{i+n})$. Moreover freeness  of the $\mbf a_\ell$'s implies \cite[Lemma 5.18]{NS06} that $ \Phi(a_1b^*_{i+1}) \cdots \Phi(a_nb^*_{i+n}) = \Phi( a b_{[i]}^*)$, where $b_{[i]} = b_{i+1} \cdots b_{i+n}$ with indices counted modulo $n$. Hence Mingo-Speicher formula is valid if and only if $\Phi^{(2)}(\hat a,\hat b^*)=\sum_{i=0}^{n-1}  \Phi( a b_{[i]}^*)=\frac 1 n \sum_{i,j=0}^{n-1} \Phi( a_{[j]} b_{[i]}^*) = n  \Phi( \hat a\hat b^*)$.
\end{proof}

For comparison, we recall the definition of ordinary 2$^{nd}$ order freeness. Let $(\mcal A, \Phi, \Phi^{(2)})$ be a 2$^{nd}$ order probability space and $\mbf a_\ell, \ell\in L$ be families in $\mcal A$. A tuple $(a_1\etc a_n)$ of elements of $\mcal A$ is called cyclically alternating and centered whenever $a_i= p_i(\mbf a_{\ell_i}) $ for indices $\ell_1 \neq \cdots \neq \ell_n \neq \ell_1$ and $\Phi(a_i)=0$ for all $i$. 

\begin{definition}The families $\mbf a_\ell$ are \emph{2$^{nd}$ order free} if they are free, and:
	\begin{enumerate}
		\item  the vector spaces $E'_n, n\geq 2,$ are orthogonal w.r.t. $\Phi^{(2)}$,  the $^*$-algebras generated by $\mbf a_\ell,\ell \in L$, are orthogonal w.r.t. $\Phi^{(2)}$;
		 \item for all $ (a_1\etc a_n), (b_1\etc b_n)$ cyclically alternating and centered, with $a=a_1\cdots a_n$, $b=b_1\cdots b_n$
	\eqa
		\Phi^{(2)}(a, b^*) =   \sum_{i=0}^{n-1}  \Phi(a_1b^*_{i+1}) \Phi(a_2b^*_{i+2}) \cdots \Phi(a_nb^*_{i+n});
		\qea
		or equivalently $ \Phi^{(2)}(p,q^*) = n \Phi(pq^*)$, for all $n\geq 2$ and all $p,q$ in 
	$$F'_n := \textrm{Span} \, \Big\{  \sum_{i=0}^{n-1} a_{i+1} \cdots a_{i+n}, \, \Big| \, (a_1\etc a_n) \  \textrm{cyclically \ alternating \ centered} \ \Big\}.$$
	\end{enumerate}
	
\end{definition}

\section{Preliminaries on $\Delta$-distributions}\label{Sec:NCprop}

This section reviews properties of $\Delta$-distributions and useful tools:
\begin{itemize}
	\item Lemma \ref{UniqSD} studies the (generalized) Schwinger-Dyson (SD) equation;
	\item Lemma \ref{Lem:Transitivity} shows that the solution of SD equation satisfies (H5);
	\item Lemma \ref{WhyA1} considers characterization questions of $\Delta$-distributions.
\end{itemize}

Whereas later on we will always consider matrices, we do not assume in this section  that $\mcal D$ is commutative. We use the two following notions of degrees. 

\begin{definition}
Let $m$ be a $\Delta$-monomial, bracketing of a word $\omega=a_1\cdots a_n$. We call $n$ the \emph{full degree of $m$}. We call \emph{ground degree of $m$} the integer $n'$ such that $m=d_0 a_{j_1}d_1 \cdots d_{n'-1}a_{j_{n'}}d_{n'}$ where $j_k\in [n]$ and $d_j$ is $\Delta$-invariant. Full and ground degrees in a sub-family of variables are defined in the obvious way.
\end{definition}

Note that $n'\leq n$ with equality when $m$ is a monomial, and $n'=0$ whenever $m$ is $\Delta$-invariant.  For example $m= y\Delta(x y )y$ has full degree 4, ground degree 2; in the variable $  x$ it has full degree 1 and ground degree 0.

\begin{lemma}\label{UniqSD} Let $\Phi_0$ be a tracial $\Delta$-invariant linear form on $\mbb C\langle\mbf y\rangle_\Delta$. There exists a unique tracial $\Delta$-invariant linear form $\Phi$ on $\mbb C\langle \mbf x, \mbf y\rangle_\Delta$ such that $\Phi_{| \mbb C\langle \mbf y\rangle_\Delta} = \Phi_0$ and $\Phi$ is a solution of the SD equation: $\Phi [ x m  ] =\sum_{ m = \ell x r}\Phi [\ell] \Phi [r]$ for any $\Delta$-monomial $m$.
\end{lemma}

\begin{proof}[Proof of Lemma \ref{UniqSD}]
Assume first that such a $\Phi$ exists. Let $m \in \mbb C\langle \mbf x, \mbf y\rangle_\Delta$ but not in $\Dxy\langle \mbf y \rangle$, i.e. $m$ is of positive ground degree in a variable $x\in \mbf x$. Then $m = m_1 x m_2$ for $\Delta$-monomials $m_1$ and $m_2$. By traciality, $\Phi(m ) = \Phi(x m_2m_1)$, so the SD equation implies that $\Phi(m)$ is function of $\Delta$-moments of smaller ground degrees in $x$. By induction $\Phi(m)$ is a polynomial in $\Phi(m')$ for $m'$ of zero ground degree in $\mbf x$, i.e. $m'\in\Dxy\langle \mbf y \rangle$.

We now assume $m\in \Dxy\langle \mbf y \rangle \setminus \mcal D^{(\mbf y)}\langle \mbf y \rangle$: it has zero ground but positive full degree in $\mbf x$. Then $m = m_1 d m_2$ for $d \in \Dxy$ with positive full degree in $x$. We can write $d= \Delta(p)$ where $p$ is not $\Delta$-invariant and is obtained by removing the outer evaluation of $\Delta$. 

By traciality, $\Delta$-invariance and the property of the conditional expectation
\eqa\label{Eq:GBC}
	\Phi(m ) = \Phi\big( \Delta(p) m_2m_1\big) =  \Phi\big(p  \Delta( m_2m_1) \big) = \Phi( \tilde m)
\qea
where $\tilde m = p \Delta(m_2m_1)$.  We call this method \emph{ground block change} (see also Example \ref{example1GBC}).  If $p$ does not have positive ground degree in $\mbf x$ we iterate the ground block change, we get a $\Delta$-monomial with positive ground degree in some variable $x$. By the previous step, we write $\Phi(m)$ as a function of $\Delta$-moments of smaller full degrees in $\mbf x$. By induction, $\Phi(m)$ is a function of $\Phi(m')$  for $m'$ of zero full degree in $\mbf x$, i.e.  $m'\in \mbb C\langle \mbf y \rangle_\Delta$. 

The conclusion is that necessary $\Phi(m) = f_m\big(\Phi(m_1) \etc \Phi(m_K)\big)$ for a polynomial function $f_m$ and $m_1\etc m_K$ in  $\mbb C\langle \mbf y \rangle_\Delta$. One checks that this unique proposal indeed provides a solution of the Schwinger-Dyson equation.
 \end{proof}
 
 \begin{remark}\label{Rk:SDconv} For each $N$ let $\Phi_N$ be a tracial $\Delta$-invariant linear form on $\mbb C\langle \mbf x, \mbf y\rangle_\Delta$, such that its restriction on $\mbb C\langle \mbf y\rangle_\Delta$ converges and such that $\Phi_N [ x m  ] =\sum_{ m = \ell x r}\Phi_N[x^2] \Phi_N [\ell] \Phi_N [r] +o(1) $ for any $\Delta$-monomial $m$. Then the above proof shows that $\Phi_N(m) = f_m\big(\Phi_N(m_1) \etc \Phi_N(m_K)\big) +o(1)$, for the same $f_m$ and $m_k$'s, and so $\Phi_N$ converges and is limit is the unique solution the SD equation.
 \end{remark}

\begin{definition}\label{def:DeltaNC} Let $m$ be a $\Delta$-monomial obtained by bracketing a word $a_1 \cdots a_n$. We define inductively a partition $\sigma(m)$ of $[n]$. If $m=a_1 \cdots a_n$, i.e. $m$ is a monomial, then $\sigma(m)$ has a single block $[n]$. Otherwise, there is a subword $\Delta(a_{i}a_{i+1} \cdots a_{i+j})$ in $m$ with no internal bracket. Then $\{i,i+1\etc i+j\}$ is a block of $\sigma(m)$. We iterate this construction on the $\Delta$-monomial obtained by deleting $\Delta(a_{i}a_{i+1} \cdots a_{i+j})$ from $m$, until no bracket remains. The blocks constructed along the process define a partition $\sigma(m)$ of $[n]$. 
\end{definition}

The partition $\sigma(m)$ is non-crossing by the nesting characterization: a partition is non crossing whenever removing an inner block produces a non-crossing partition, see \cite[Remark 9.2]{NS06}. For any $\Delta$-monomial $m$ bracketing of a word $a_1 \cdots a_n$ and any block $B=\{i_1\etc i_k\}$ of $\sigma(m)$, we denote the monomial $m_B = a_{i_1} \cdots a_{i_k}$. A simple induction reasoning in the above construction shows that a family $\mbf y$ satisfies (H5) whenever for any  $\Delta$-monomial $m \in \mbb C\langle \mbf y \rangle_\Delta$,
\begin{enumerate}
	\item[{(H5')}]		$\qquad \Phi(m) = \prod_{B \in \sigma(m)} \Phi(m_B);$
\end{enumerate}
We use this observation, the ground block change method and the SD equation to justify the last  statement in Theorem \ref{MainTh1} that we reformulate in next lemma:

\begin{lemma}\label{Lem:Transitivity} Let $\Phi$ be a tracial $\Delta$-invariant linear form on $\mbb C\langle  \mbf x, \mbf y \rangle_\Delta$ solution of the SD equation \eqref{SDeq}. If  (H5') holds on $\mbb C\langle  \mbf y \rangle_\Delta$ then (H5') also holds on $ \mbb C\langle  \mbf x, \mbf y \rangle_\Delta$.
\end{lemma}

Let $ m $ be a $\Delta$-monomial, bracketing of the word $a_1 \cdots a_n$. Writing $m =d_0 a_{j_1}d_1 \cdots d_{n'-1}a_{j_n}d_{n'}$ where $j_k\in [n]$ and $d_j \in \Dxy$, we then distinguish the block $\{j_1,j_2 \etc j_n\}$ of $\sigma(m)$ that we call the \emph{ground block} of $m$ (its cardinal is the ground degree of $m$). The other blocks of $\sigma(m)$ are called the \emph{paddle blocks}.

\begin{proof}[Proof of Lemma \ref{Lem:Transitivity}] We set $\tilde \Phi(m) = \prod_{B \in \sigma(m)}\Phi(m_B)$, for any $\Delta$-monomial $m$ and show $\tilde \Phi=  \Phi$ by proving it is solution of the Schwinger-Dyson equation. By assumption, the functions coincide in $\mbb C\langle \mathbf y \rangle_{\Delta}$, so let $m\in\mbb C\langle \mathbf x, \mathbf y \rangle_{\Delta}\setminus \mbb C\langle \mathbf y \rangle_{\Delta}$. Recall the  ground block change method, illustrated in \eqref{Eq:GBC}: there is a $\Delta$-monomial $\tilde m$, such that $\Phi(m) = \Phi(\tilde m)$ and $\tilde \Phi(m) = \tilde \Phi(\tilde m)$, with positive ground degree in a variable $x$. Hence we can always assume that $m$ has positive ground degree in a variable $x$. By traciality, we can assume $m = xm'$ for $x\in \mathbf x$ and $m' \in \mbb C\langle \mathbf x, \mathbf y \rangle_{\Delta}$. 

For any $\Delta$-monomial $p$, we set $p_{GB}$ the monomial associated to the ground block of $p$. Then we have $m_{GB}= x m'_{GB}$, and by Schwinger-Dyson equation for $\Phi$
	\eq
		\tilde \Phi(m)  
		=\sum_{m'_{GB} = \ell_{GB} x r_{GB}} \Phi(\ell_{GB} ) \Phi(r_{GB}) \prod_{\substack{ B'\in \sigma(m) \\ B'\neq GB(m)}}\Phi(m_B),
	\qe
where $GB(m)$ is the ground block of $m$. But for any decomposition of $m'_{GB}$ into a product $\ell_{GB} x r_{GB}$ corresponds a decomposition $m'=\ell x r$ where $\ell$ and $r$ are $\Delta$-monomials whose ground monomials are $\ell_{GB}$ and $r_{GB}$ respectively. The set of paddle blocks of $m'$ is the union of those of $\ell$ and of $r$. Hence
	$$\tilde \Phi(xm')  =\sum_{m' = \ell  x r}  \tilde \Phi(\ell ) \tilde \Phi(r),$$
and so $\tilde \Phi$ satisfies the Schwinger-Dyson, and by Lemma \ref{UniqSD} it coincides with $\Phi$. 
\end{proof}

	Recall that given families $\mbf a_\ell$, $\ell\in L$, in Definition \ref{def:An} we denote  by $\mcal A_1$ the smallest space such that $\mcal \mcal Alg(\mbf a_\ell, \Delta( \mcal A_1) ) \subset \mcal A_1$, $\forall \ell \in L$. The inductive construction of Definition \ref{def:DeltaNC} shows that $\mcal A_1$  is generated by $\Delta$-monomials $m$ evaluated in the $\mbf a_\ell$'s in such a way each $m_B$ is evaluated in one family $\mbf a_\ell$, $\forall B \in \sigma(m)$, i.e. each $m_B$ is \emph{non-mixing}. Recall   we also denote by $\mcal A_n, n\geq 2,$ the space of cyclically $\mcal D^{(\mbf a)}$-alternating products of $\Delta$-centered elements in the $\mbf a_\ell$'s. The following property is crucial to consider the characterization of $\Delta$-distributions in terms of marginals.

\begin{lemma}\label{WhyA1} With notations as above, for any tracial $\Delta$-invariant linear map  $\psi$ on $\mcal A$ and any  $m\in \mathbb C\langle \mbf a \rangle_\Delta$,  $\psi(m)$ is a linear combination of $\psi(q)$ for $q \in \mcal A_n$, $n\geq 1$.
\end{lemma}

Hence if $\mbf a_\ell, \ell \in L$ are free over $\mcal D$ then the $\Delta$-distribution $\Phi$ vanishes on $\mcal A_n$ for $n\geq 2$; it only remains to determine $\mcal A_1$. On the other hand if the collection of all $\mbf a_\ell$ satisfies (H5) then $\Phi$ is determined on $\mcal A_1$ by its restriction on monomials on each family $\mbf a_\ell$;  using (H5) again, if moreover the  $\mbf a_\ell$'s are free in the ordinary sense, then $\Phi$ also vanishes on $\mcal A_n$ for $n\geq 2$. By Lemma \ref{WhyA1}, the joint $\Delta$-distribution of the $\mbf a_\ell$'s is hence determined by the marginal $\Delta$-distribution.

Assume moreover that $\mbf a_\ell, \ell \in L$ are 2$^{nd}$ order $\Delta$-free. By Lemma \ref{WhyA1}  and by the orthogonality conditions, the 2$^{nd}$ order $\Delta$-distribution is determined 
\begin{itemize}
	\item \emph{(universal part)} on each $\mcal A_n,n\geq 2$ by the marginal 1$^{st}$ order distribution, via Mingo-Speicher formula and ordinary freeness;
	\item \emph{(non-nuniversal part)} the marginal 1$^{st}$ and 2$^{nd}$ order distribution via Leibnitz formula on $\mcal A_1$. 
\end{itemize}
\begin{proof}[Proof of Lemma \ref{WhyA1}]
Let $m$ be a $\Delta$-monomial in the variables $\mbf a_\ell,\ell\in L$ which is not in $\mcal A_1$: there is a block $B\in \sigma(m)$ such that $m_B$ in a $\Delta$-monomial in at least two different $\mbf a_\ell$'s. We set $\beta(m)\geq 2$ the sum of degrees of monomials $m_B$ in at least two different families of variables $\mbf a_\ell$'s, for $B\in \sigma(m)$, and we reason by induction on $\beta(m)$.

By ground block change we can assume that $B$ is the ground block of $m$. Then we write 
$m= r_0  s_1 r_1   \cdots s_n r_n $, where $n\geq 2$, each $s_k \in  \mcal D^{(\mbf a)}  \langle \mbf a_{\ell_k}\rangle$ is of the form $s_k = a m'_k = m''_k a'$ for $a,a'\in \mbf a_{\ell_k}$, and each $r_k$ is in $\mcal D^{(\mbf a)}$ i.e. is $\Delta$-invariant, and $\ell_k \neq \ell_{k+1}$.  Traciality implies that $\psi( m) = \psi( s_1r_1 \cdots s_n r_n r_0   )$, so we can assume $r_0=1$. If $j_n=j_1$ then with $r'_1=s_n r_ns_1$ we have $\psi( m ) = \psi( s'_1r_1 \cdots s_{n-1} r_{n-1} )$. So we can assume $j_n\neq j_1$. 

We operate to the usual reduction in the free product construction: we set $\tilde s_k =s_k - \Delta(s_k)$,  and $q=\tilde s_1   r_1 \cdots \tilde s_n   r_n$. Then $m = q+ q'$, where $q\in \mcal A_n$ and $  q'$ is a  linear combination of $\Delta$-monomials $m'$ such that $\beta(m')<\beta(m)$. By induction, $\psi(m)$ is a linear combination of $\psi(m')$ where $m'\in \bigcup_{n\geq 1} \mcal A_n$ 
\end{proof}

\section{First-order convergence: proof of Theorem 2.1}\label{Sec:FirstOrder}
	Let $\mbf X_N=(X_\ell)_{\ell\in L}$ and $\mbf Y_N=(Y_j)_{j \in J}$ be as in Theorem \ref{MainTh1}, and $\mbf x=(x_\ell)_{\ell\in L},$ $\mbf y=(y_j)_{j \in J}$ be indeterminates. For each Wigner matrix $X_N=\frac1 {\sqrt N}\big( {x_{ij}}   \big)_{ij}$, we assume that $\esp\big[|x_{12}|\big]=1$ and we assume that $\mbf Y_N$ is closed by adjoint. We prove the convergence of 
	$\Phi_N(q):=\esp\big[ \frac 1 N \Tr \, q(\mbf X_N, \mbf Y_N) \big]$.  Convergence almost sure and in probability follow easily from variance estimates implied by Theorem \ref{MainTh2} and Borel-Cantelli lemma, see for instance the proof of \cite[Theorem 1.13]{Gui09}. Since the trace is tracial and invariant under $\Delta$, we can assume $q$ is $\Delta$-invariant. 
\begin{enumerate}
	\item Next section in Proposition \ref{Prop:AsympIndep} we recall an asymptotic formula based on \emph{traffic independence}. It does not imply the convergence of $\Phi_N(q)$ since it involves quantities bounded but not necessarily convergent.
	\item In  Section \ref{Sec:SDeq} we prove that $\Phi_N$ satisfies asymptotically the SD equation. This implies Theorem \ref{MainTh1} thanks to the preliminary results.
\end{enumerate}

\subsection{Recall of asymptotic first-order formulas}\label{Sec:FirstOrderRecall}

We call \emph{test graph} a triplet $T=(V,E,\gamma)$, where $(V,E)$ is a finite directed graph $(V,E)$ and $\gamma$ is a labelling map from $E$ to a label set. The graph may have loops and multiple edges, so that $V$ is a set and $E$ is a multiset of ordered pair of elements of $V$. The map $\gamma$ indicates that the edge $e$ is labeled by the variable $x_{\gamma(e)}$. Test graphs are given with a partition $E=E_X\sqcup E_Y$ of the edge set: a labelling $\gamma(e)$ of $e\in E_X$ is an element of $L$ which refers to the variable $x_{\gamma(e)}$, we say with small abuse that $e$ is \emph{labeled by the Wigner matrix $X_{\gamma(e)}$}; a labelling $\gamma(e)$ of $e\in E_Y$ belongs to $J$ and refers to $y_{\gamma(e)}$, $e$ is \emph{labeled by the deterministic matrix $Y_{\gamma(e)}$}. 

The \emph{trace of a test graph $T=(V,E,\gamma)$ in the matrices $\mbf X_N, \mbf Y_N$} is the scalar
	\begin{eqnarray}\label{Def:Trace}
	\lefteqn{\Tr \big[ T(\mbf X_N, \mbf Y_N) \big] } \nonumber \\
	 & =  & \sum_{\phi: V\to [N]} \prod_{e=(v,w)\in E_X} X_{\gamma(e)}\big( \phi(w), \phi(v) \big) \times \prod_{e = (v,w) \in E_Y} Y_{\gamma(e)}\big( \phi(w), \phi(v)\big).
	\end{eqnarray}

\begin{definition} A \emph{well-oriented (w.o.) simple cycle} is a cycle with no repetition of vertices whose edges follow a same orientation. A \emph{w.o. cactus} is a test graph such that each edge belongs to a unique simple cycle and all simple cycles are well-oriented. 
\end{definition}
Let $m$ be a $\Delta$-monomial obtained by bracketing a word $a_1 \cdots a_n$, where $a_i \in \mbf x \cup \mbf y$. We associate inductively a w.o. cactus $T_m$ in the similar way we have constructed the non crossing partition $\sigma(m)$. Let $T^{(0)}$ be the simple directed cycle with $n$ edges labeled  $\cdots \overset{a_{i-1}}\leftarrow \cdot  \overset{a_i}\leftarrow  \cdot \overset{a_{i+1}}\leftarrow  \cdots $.  If $m=a_1 \cdots a_n$, i.e. $m$ is a monomial, then $T_m = T^{(0)}$. Otherwise, there is a subword $\Delta(a_{i}a_{i+1} \cdots a_{i+j})$ in $m$ with no internal bracket. Let $T^{(1)}$ be the graph obtained from $T^{(0)}$ by identifying the target of the $i$-th edge and the source of $(i+j)$-th one. Then we iterate this construction on the cycle of $T^{(1)}$ that do not contain the edges labeled $a_{i},a_{i+1} \cdots a_{i+j}$. We thus construct a sequence $T^{(k)},k=0,1,\dots$ of graphs until a step $k_0$ for which there is not bracketed expression remaining. We then set $T^{(k_0)}=T_m$. The edges are canonically labeled by the variables of the monomial. 

By construction, $T_m$ is a w.o. cactus and its cycles are correspondence with the blocks of the non-crossing partition $\sigma(m)$.  One checks that the trace of $m$ is the trace of the graph $T_m$ in the matrices in the sense of \eqref{Def:Trace}, namely
	\eqa\label{Eq:TrTrTest}
		\Tr \, m(\mbf X_N, \mbf Y_N) = \Tr  \big[T_m(\mbf X_N, \mbf Y_N) \big].
	\qea
For any test graph $T$, the \emph{injective trace of $T$ in the matrices $\mbf X_N, \mbf Y_N$} is the quantity $\Tr^0 \big[ T(\mbf X_N, \mbf Y_N) \big] $ defined as in \eqref{Def:Trace} but where the sum is restricted to injective maps $\phi:V\to [N]$.
Let $T=(V,E,\gamma)$ be a test graph. For any $\pi \in \mcal P(V)$, let $T^\pi$ be the graph obtained by identifying vertices in a same block of $\pi$. We say that $T^\pi$ is a quotient of $T$. Then we have relation \cite[Section 2.1]{Mal11}: for any test graph $T$ 
	\eqa\label{Eq:TrTestTrInj}
		 \Tr \, T(\mbf X_N, \mbf Y_N) = \sum_{\pi \in \mcal P(V) } \Tr^0  \big[T^\pi(\mbf X_N, \mbf Y_N) \big].
	\qea

In this paragraph and in Definition \ref{GCC}, we fix a $\Delta$-monomial $m$ and denote $T_m =T= (V,E,\gamma)$ its w.o. cactus. For $\pi \in \mcal P(V)$,  let $T^\pi_X = (V^\pi_X,E^\pi_X, \gamma^\pi_X)$ be the subgraph of $T^\pi$ whose edges are labeled by Wigner matrices. Let $T^\pi_Y = (V^\pi_Y,E^\pi_Y, \gamma^\pi_Y)$ be defined similarly with deterministic matrices and let $C_Y(T_m^\pi)$ be the set of connected components of $T^\pi_Y$. We define the weight associated to $\mbf X_N$
	\eqa\label{Eq:OmegaX1}
		\omega_X^{(1)}(\pi) = \esp\Big[ \prod_{e=(v,w) \in E^\pi_X} \sqrt{N} X_{\gamma(e)}\big( \psi(w), \psi(v) \big) \Big]
	\qea
for any choice of $\psi:V^\pi_X\to [N]$ injective (permutation invariance of $\mathbf X_N$ implies that this quantity does not depend on $\psi$), and the weight associated to $\mbf Y_N$
	\eqa\label{Eq:OmegaY}
	\omega_Y(\pi) = \frac 1 {N^{|\mcal C_Y(T^\pi)|}} \Tr^0 \big[ T^\pi_Y(\mbf Y_N) \big].
	\qea
By \eqref{Eq:TrTrTest}, \eqref{Eq:TrTestTrInj} and the permutation invariance of $\mbf X_N$, we have \cite{MMPS}
	\eqa\label{Eq:FirstOrderExact}
		\Phi_N(m) :=\esp\Big[ \frac 1 N \Tr \, m(\mbf X_N, \mbf Y_N) \Big] =  
		\sum_{\pi \in \mcal P(V)} N^{q^{(1)}(\pi)}\omega^{(1)}_X(\pi) \omega_Y(\pi),
	\qea
where $q^{(1)}(\pi)=-\frac {|E_X|}2 -1 + |\mcal C_Y(T)|$. Mingo Speicher's bounds for the trace of graphs \cite{MS12} then allows us to replace $N^{q^{(1)}(\pi)}$ by an indicator function for $N$ large. 
	
\begin{definition}\label{GCC}
\begin{enumerate}
	\item A tree is a connected graph (directed or not) such that the removal of an arbitrary edge always disconnects the graph.
	\item Fix $\pi \in \mcal P(V)$. The \emph{graph of deterministic components of $T^\pi$} is the undirected graph  $\mcal {GDC}(T^\pi)= (\mcal V, \mcal E)$, where
\begin{itemize}
	\item the vertex set $\mcal V$ is the disjoint union of the vertex set of $T^\pi$ and of the set $\mcal  C_{Y}(T^\pi)$ of connected components of the subgraph of $T^\pi$ of edges labeled by deterministic matrices,
	\item the edge set $\mcal E$ is the union of $E^\pi_X$ and of the set of pairs $\{v,S\}$ where $c\in V$ belongs to the subgraph $S\in \mcal C_{Y}(T^\pi)$. 
\end{itemize}
	\item We say that $T^\pi$ is of \emph{double-tree type} and denote $T^\pi\in \boldsymbol{ \mcal{DT}}$ if the graph $\overline{\mcal {GDC}}(T^\pi)$ obtained from $\mcal {GDC}(T^\pi)$ by forgetting edge multiplicity is a tree, the edges of $\mcal {GDC}(T^\pi)$ labeled in $\mbf x$ have multiplicity 2, and edges forming a group of multiplicity 2 are labeled by a same variable $x$.
\end{enumerate}
\end{definition}

\begin{proposition}\label{Prop:AsympIndep} Under (H1) and (H3), we have the first-order asymptotic formula: for all $\Delta$-monomials $m$ and all partitions $\pi$ of the vertices of  $T_m$,
$$	 N^{q^{(1)}(\pi)}\omega^{(1)}_X(\pi) \omega_Y(\pi)  =  \one\big( T_m^\pi \in \boldsymbol{ \mcal{DT}}\big) \prod_{ C \in \mcal C_Y(T^\pi)} \frac 1 N \Tr^0 \big[ C(\mbf Y_N) \big] +o(1),$$
and these quantities are uniformly bounded when $N$ goes to infinity. 
\end{proposition}
The heuristic of the proposition is presented in \cite{MMPS} and the result is an immediate consequence of the asymptotic traffic independence theorem of \cite{Mal17} with the limiting traffic distribution of \cite[Section 3.1]{Mal11}. 

\subsection{Proof of the asymptotic SD equation}\label{Sec:SDeq}
We recall that we denote $\Phi_N: m \mapsto \esp\big[ \frac 1 N \Tr \, m(\mbf X_N, \mbf Y_N) \big] $. By Remark \ref{Rk:SDconv}, to prove Theorem \ref{MainTh1} we shall prove: for all $\Delta$-monomial $m$ and all $x\in \mbf x$, then $\Phi_N [ x m  ] =\sum_{ m = \ell x r} \Phi_N(\ell) \Phi_N(r)+o(1),$ where the sum is over all decompositions of $m$ as a product of $\Delta$-monomials.

We shall use the following general observation. Let $m \in \mathbb C\langle \mbf x, \mbf y \rangle_\Delta$ be $\Delta$-invariant. With same notations as in previous section, we denote $T=T_m$ the w.o. cactus associated to $m$ and $V=V_m$ its vertex set. Let $\pi\in \mcal P(V)$ such that the quotient $T^\pi$ is of double tree type. By definition, an edge $e$ labeled $x$ in $T^\pi$ is hence grouped with another edge $e'$ of same label to form a double edge. We say in short that $e$ and $e'$ are \emph{twin edges}. We claim that \emph{twin edges  $e$ and $e'$ in $T^\pi$ always belong to a same simple cycle in $T$}. Indeed, since $T^\pi\in \boldsymbol{ \mcal{DT}}$, the graph $T^\pi_{\setminus e,e'}$ obtained by removing $e$ and $e'$ from $T^\pi$ is disconnected. Contrariwise, if $e,e'$ are not in a same cycle of $T$, the graph $T_{\setminus e,e'}$, obtained by removing $\{e,e'\}$ from $T$, is connected. But $T^\pi_{\setminus e,e'}$ is a quotient of $T_{\setminus e,e'}$ and taking a quotient cannot increase the number of connected components. The same connectivity argument shows that \emph{twin edges have opposite orientation} since removing $e,e'$ must disconnect $T^\pi$: the source of $e$ is the target of $e'$ and reciprocally.

Let now prove the asymptotic SD equation. Let $m$ be a $\Delta$-monomial of the form $m=xm'$, and let $T=T_m$ be its w.o. cactus. Denote by $e$ the edge of $T$ associated to the first letter $x$ of the word of $xm'$. The edges $e'\neq e$ in the same cycle as $e$ with label $x$ are in correspondence with the decompositions $m' = \ell x r$. 

We fix the edge $e'$ which forms a twin edge with $e$, and so a decomposition $m'=\ell x r$. Let $\pi\in \mcal P(V)$ such that $T^\pi\in \boldsymbol{ \mcal{DT}}$ and $e,e'$ are twin edges. Then $T^\pi$ is a quotient of the graph $T_{e\sim e'}$ obtained from $T$ by identifying the source of $e$ and the target of $e'$ and reciprocally. Note that $T_{e \sim e'}$ consists in the two disjoint graphs $T_\ell$ and $T_r$ (defined as $T_m$ with $\ell,r$ replacing $m$) linked by the double edge $\{e, e'\}$.  
Seeing $T^\pi$ as a quotient of $T_{e\sim e'}$, there is no identification of vertices of different subgraphs $T_\ell$ and $T_r$, by the same connectivity argument as before (removing $e,e'$ must disconnect $T^\pi$). Hence the set of $\pi \in \mcal P(V)$ such that $T^\pi\in \boldsymbol{ \mcal{DT}}$ and $e,e'$ are twin edges is in correspondence with the set of couples  $(\pi_\ell,\pi_r) \in \mcal P(V_\ell) \times \mcal P(V_r)$ such that $T_\ell^{\pi_\ell},T_r^{\pi_r}\in \boldsymbol{ \mcal{DT}}$ ($V_\ell, V_r$ denote the vertex sets of $T_\ell,T_r$). So we have 
	\eq
		\lefteqn{\sum_{ \substack{ \pi \in \mcal P(V) \\ e \sim e'}} \one\big( T_m^\pi \in \boldsymbol{ \mcal{DT}}\big) \prod_{ C \in \mcal C_Y(T^\pi_m)} \frac 1 N \Tr^0 \big[ C(\mbf Y_N) \big]}\\
		&  = & \prod_{s \in \{ \ell, r\}} \sum_{ \pi_s \in \mcal P(V_s) } \one\big( T_s^\pi \in \boldsymbol{ \mcal{DT}}\big) \prod_{ C \in \mcal C_Y(T_s^{\pi_s})} \frac 1 N \Tr^0 \big[ C(\mbf Y_N) \big].
	\qe
By Proposition \eqref{Prop:AsympIndep}, we get the claimed result:
	$$\esp\Big[ \frac 1 N \Tr \big[ m(\mbf X_N,\mbf Y_N) \big] \Big] = \sum_{ m' = \ell x r } \esp\Big[ \frac 1 N \Tr \big[ \ell(\mbf X_N,\mbf Y_N) \big] \Big]  \esp\Big[ \frac 1 N \Tr \big[ r(\mbf X_N,\mbf Y_N) \big] \Big] +o(1).$$

\section{Second-order convergence}\label{Sec:SecondOrder}

\subsection{Recall of asymptotic formulas}\label{Sec:RecallSecondOrder}

We consider $\mbf X_N=(X_\ell)_{\ell\in L}$ and $\mbf Y_N=(Y_j)_{j \in J}$ as in Theorem \ref{MainTh2}. As in the previous section, we assume the Wigner matrices normalized and the family of deterministic matrices closed by adjoint. For two $\Delta$-polynomials $q_1$ and $q_2$, we denote $\Phi_N^{(2)}(q_1,q_2)$ the quantity
	$$ \esp\bigg[   \Big( \Tr \, q_1(\mbf X_N, \mbf Y_N) - \esp\big[ \Tr \, q_1(\mbf X_N, \mbf Y_N) \big] \Big)  \Big( \Tr \, q_2(\mbf X_N, \mbf Y_N) - \esp\big[ \Tr \, q_2(\mbf X_N, \mbf Y_N) \big] \Big)\bigg].$$

\begin{enumerate}
	\item We recall from \cite{MMPS} asymptotic formulas for the 2$^{nd}$ order distribution on $\Delta$-monomials, in a similar way as for first-order formulas. They involve the fourth moment of the normalized entries of Wigner matrices. 
	\item The novelty compared to the first-order convergence is that we are now considering the $\Phi_N^{(2)}$ on alternating products of reduced elements. We show in Section \ref{Sec:TraceApre} how this discards the dependence in the fourth moment. We also prove there the orthogonality of the spaces $E_n$, $n\geq 1$. 
	\item Section \ref{Sec:ProofMainTh4} is devoted to the proof of Mingo-Speicher formula. In Lemma \ref{SecondOrderSD2}, we prove that $\Phi^{(2)}$ satisfies a variant of SD equation. Comparing them with similar equations satisfied by $\Phi$, we deduce (ii).
	\item Section \ref{Sec:ProofMainTh3} is devoted to the proof of (iii), i.e. the Leibniz property. 
\end{enumerate}

Without loss of generality, we assume that $\mbf Y_N$ contains the identity matrix and is closed by $\Delta$-monomial ($q(\mbf Y_N) \in \mbf Y_N$ for each $q \in \mbb C \langle \mbf y \rangle_\Delta$). We first consider two words of the form 
	\eq
		m_1 &=& x_{\ell_1}y_{j_1}\cdots x_{\ell_p}y_{j_p},\\
		m_2 &=  &x_{\ell_{p+1}}y_{j_{p+1}} \cdots x_{\ell_{p+q}} y_{j_{p+q}}.
	\qe
We denote by $T_1=T_{m_1}$ and $T_2=T_{m_2^*}$ the test graphs (c.f. Section \ref{Sec:FirstOrderRecall}) associated to $m_1$, ${m_2^*}$ respectively, and 
we set $T=T_{m_1}\sqcup T_{m_2^*}= (V,E,\gamma)$ which hence consists of the disjoint union of the two cycles.

Then we consider two $\Delta$-monomials $p_1$ and $p_2$  obtained by bracketing the expressions $m_1$ and $m_2$ respectively. We denote by $T_{p_1}$ and $T_{p_2^*}$ the w. o. cacti associated to $p_1$ and ${p_2}^*$  respectively: they are quotients of $T_{m_1}$ and $T_{m_2^*}$ respectively. We denote by $\nu_i$ the partition of the vertex set of $T_i$ such that $T_{p_1}=T_1^{\nu_1}$, $T_{p_2^*}=T_2^{\nu_2}$. We set $\nu = \nu_1\sqcup \nu_2 \in \mcal P(V)$, so that $T_{p_1}\sqcup T_{p_2^*}=T^\nu$. 

Let $\pi\in \mcal P(V)$ be an arbitrary partition and consider the quotient $T^\pi$ of the union of cycles $T$. The restriction of $\pi$ on the vertex sets of each cycle defines two subgraphs $T_1^\pi$ and $T_2^\pi$ of $T^\pi$. We set $T^\pi_{X,i}=(V_{X,i}^\pi,E_{X,i}^\pi, \gamma_{X,i}^\pi)$, $i=1,2$, the subgraph of $T_{i}^\pi$ whose edges are labeled by $\mbf x$, and $T_Y^\pi$ the subgraph of $T^\pi$ whose edges are labeled by $\mbf y$. We set the weights of 2$^{nd}$ order associated to Wigner matrices
	\eqa\label{Eq:OmegaX2}
		\omega_X^{(2)}(\pi) &=& \esp\Big[ \prod_{\substack{e=(v,w)\\ \in E^\pi_{X,1}\cup E^\pi_{X,2} }} \sqrt{N} X_{\gamma_X(e)}\big( \psi(w), \psi(v) \big) \Big]\\
		\nonumber & & \ \  - \prod_{i=1,2}\esp\Big[  \prod_{\substack{e=(v,w)\\ \in E^\pi_{X,i} }} \sqrt{N} X_{\gamma_X(e)}\big( \psi(w), \psi(v) \big) \Big]
	\qea
for any injective map $\psi: V^\pi\to [N]$,
and $\omega_Y(\pi)=  \frac 1 {N^{|\mcal C_Y(T^\pi)|}}  \Tr^0 \big[ T^\pi_Y(\mbf Y_N) \big]$ defined as in \eqref{Eq:OmegaY} with the new definition of $T=T_{ m_1}\sqcup T_{m_2^*}$. Then, by permutation invariance of $\mbf X_N$, we have the formula, proved in \cite[Section 2.1-2.4]{MMPS},
	\eqa \label{Eq:SeparationTermsMonomials}
		\Phi_N^{(2)}(m_1,m_2^*) = \sum_{\pi \in \mcal P(V)} N^{q^{(2)}(\pi)} \omega^{(2)}_X(\pi) \omega_Y(\pi),
	\qea
where $q^{(2)}(\pi)=-\frac{|E_1^\pi| + |E_2^\pi|}2 +|\mcal C_Y(T^\pi)|$ and $\mcal C_Y(T^\pi)$ is the set of connected components of $T_Y^\pi$. Denote by $\mcal P^\nu(V) \subset \mcal P(V)$ the set of partitions $\pi \geq \nu$, so $\pi \in \mcal P^\nu(V)$ whenever $T^\pi$ is a quotient of $T_{p_1} \sqcup T_{p_2^*}$. The same proof as for \eqref{Eq:SeparationTermsMonomials} yields 
	\eqa \label{Eq:SeparationTerms}
		\Phi_N^{(2)}(p_1,p_2^*) = \sum_{ \pi \in \mcal P^\nu(V)} N^{q^{(2)}(\pi)} \omega^{(2)}_X(\pi) \omega_Y(\pi).
	\qea
In  \cite[Section 2.1-2.4]{MMPS}, we give an asymptotic formula $N^{q^{(2)}(\pi)} \omega^{(2)}_X(\pi) \omega_Y(\pi)$ for any $\pi \in \mcal P(V)$. We can hence use it to compute $\Phi_N^{(2)}(p_1,p_2^*)$.

\begin{definition}
\begin{enumerate} Let $T^\pi$ be a quotient of $T=T_{ m_1}\sqcup T_{m_2}$.
	\item We say that $T^\pi$ is of \emph{double-unicyclique type} and denote $T^\pi \in \boldsymbol{ \mcal{DU}}$ if 
	\begin{itemize}
		\item the graph $\overline{\mcal {GDC}}(T^\pi)$, given from $\mcal {GDC}(T^\pi)$ by forgetting edge multiplicity (multip.), has a unique simple cycle, the edges of $\mcal {GDC}(T^\pi)$ labeled in $\mbf x$ have multip. 2, and twin edges are labeled by a same variable $x$ and have opposite orientation,
		\item each graph $\overline{\mcal {GDC}}(T^\pi_{i}), i=1,2$, has a unique cycle, edges labeled $\mathbf x$ on the cycle are of multip. 1, the other are of multip. 2.
		\item $T^\pi$ is obtained from $T^\pi_{1}$ and $T^\pi_{2}$ by identifying the 2 cycles.
	\end{itemize}
	
	\item We say that $T^\pi$ is of \emph{4-2 tree type} and denote $T^\pi \in \boldsymbol{ \mcal{FT}}$ if 
	\begin{itemize}
		\item the graph $\overline{\mcal {GDC}(T^\pi)}$ is a tree, all edges of $\mcal {GDC}(T^\pi)$ labeled in $\mbf x$ have multip. 2 but one group of edges have multip. 4, and edges forming a group of multip. 2 or 4 are labeled by a same variable $x$,
		\item the graphs $T^\pi_{1}$ and $T^\pi_{2}$ are of double tree type. 
		\item $T$ is obtained from $T^\pi_{1}$ and $T^\pi_{2}$ by forming a group of edges of multip. 4 from a group of edges of multip. 2 in $T^\pi_{1}$ and in $T^\pi_{2}$. 
	\end{itemize}
\end{enumerate}
\end{definition}

\begin{proposition}\label{Prop:SecondAsympIndep} \cite[Section 3]{MMPS} Under Hypotheses (H1), (H2) and (H4)
	\eqa
		\nonumber \lefteqn{ N^{q^{(2)}(\pi)} \omega^{(2)}_X(\pi) \omega_Y(\pi) }\\
		\label{Eq:2ndAsympIndep}&=&  \one\big(T^\pi \in \boldsymbol{ \mcal{DU}}\cup \boldsymbol{ \mcal{FT}} \big)\omega_X^{(2)}(\pi) \prod_{ C \in \mcal C_Y(T^\pi)} \frac 1 N \Tr^0 \big[ C(\mbf Y_N) \big] +o(1).
	\qea
If moreover $\Phi^{(2)}_N(p_1,p_2^*)$ converges for any $p_1,p_2$, then the process $\mbf Z_N$ of Theorem \ref{MainTh2} converges to a Gaussian process.
\end{proposition}

If $T^\pi \in  \boldsymbol{ \mcal{DU}}$ is of double-unicyclic type (with unicycle size $\geq2$) the normalization of Wigner matrices implies $\omega_X(\pi)=1$. Contrarywise, if $T^\pi \in  \boldsymbol{ \mcal{FT}}$ then $\omega_X(\pi)$ depends on the fourth moment of the entry of a Wigner matrix.

\subsection{Traces of alternating products of reduced elements}\label{Sec:TraceApre}

We first introduce notations in order to compute $\Phi^{(2)}$ for generic $\Delta$-polynomials. 
Notations $m_i,p_i$ and $\nu$ are before. For the sequel, we fix $p_1$ and $p_2$ of  form
\eq
	p_1 &=&  s_1 r_1  \cdots s_{n_1} r_{n_1} ,\\
	p_2 &=&  s_{n_1+1} r_{n_1+1}  \cdots s_{n_1+n_2} r_{n_1+n_2}  .
\qe
where $n_1,n_2\geq 1$ and we assume the following. Each $s_k$ belongs to $\Dxy\langle x_{\ell'_k}\rangle$ for some variable $x_{\ell'_k}$ and is of the form $s_k = x_{\ell'_k} s'_k = s''_k x_{\ell'_k}$. Each $r_k$ is in $\Dxy\langle \mbf y\rangle$; if moreover $r_k$ is $\Delta$-invariant, then the neighboring words are labeled by different variables, with the usual cyclic convention i.e. $\ell'_k \neq \ell'_{k+1}$ (if $k\notin\{ n_1, n_1+n_2\}$) or  $\ell'_{n_1}  \neq \ell'_{1}$ (if $k=n_1$) or $\ell'_{n_1+n_2}  \neq \ell'_{n_1+1}$ (if $k=n_1+n_2$).  We then define 
\eqa\label{Lesq1q2}
	q_1&  = &   \tilde s_1 \tilde r_1  \cdots \tilde s_{n_1} \tilde r_{n_1}\\
	q_2&  =  &\tilde s_{n_1+1}\tilde  r_{n_1+1}  \cdots \tilde s_{n_1+n_2} \tilde r_{n_1+n_2}
\qea
 as follow:
\begin{itemize}
	\item if $n_1\geq 2$ or $n_1=1$ and $r_1$ is not $\Delta$-invariant, then $ \tilde s_k = s_k - \Delta(s_k)$, $\tilde r_k = r_k-\Delta(r_k)$ if $r_k$ not $\Delta$-invariant and $\tilde r_k=r_k$ otherwise, $\forall k=1\etc n_1$;
	\item if $n_1=1$ and $r_1$ is $\Delta$-invariant, we simply set $q_1=p_1$;
	\item the similar definition holds for $q_2$. 
\end{itemize}
Note that in the first case $q_1 \in E_{n}$ for an $n\geq 2$: it is a cyclically $\mcal D^{(\mbf a)}$-alternating product of $\Delta$-centered elements. In the second case $q_1\in \Dxy\langle x_{\ell'_1}\rangle$: its ground monomial is a power of $x_{\ell'_1}$  (see the resemblance with the definition of $E_1$).

In Lemma \ref{Lem:Reduceness} below, we give an expression of $\Phi_N^{(2)}(q_1,q_2^*)$ obtained by ruling out certain terms in the expression \eqref{Eq:SeparationTerms} for $(p_1,p_2^*)$. Since $p_1=s_1 r_1  \cdots s_{n_1} r_{n_1}$ is not $\Delta$-invariant, its partition $\sigma(p_1)$ has a ground block. This corresponds to a cycle of $T_{p_1}$ that we call the \emph{ground cycle}. The variables in the ground block of $p_1$ appears as labels on the ground cycle of $T_{p_1}$. We call \emph{arcs} the maximal connected subsets of edges of the ground cycle that are labeled by letter of a single subword $r_k$ or $s_k$. There is hence a first arc for $s_1$ if it is not $\Delta$-invariant, there is always an arc for $r_2$ and so on.

We assume $q_1 \in E_{n}$ for some $n\geq 2$, or equivalently $n_1\geq 2$ or $n_1=1$ and $r_1$ is not $\Delta$-invariant. Hence the ground cycle is the union of at least two arcs. We denote by $(v_k,w_k), k\in \mcal I_1,$ the extremities of arcs. More precisely, we set $v_k$ the target of the last (w.r.t. the direct orientation) edge of the arc or $s_k$, and $w_{k}$ the source of its first edge. For each $k\in [n_1]$ such that $r_k$ is not $\Delta$-invariant, we define similarly the extremities $v_{n_1+k}$ and $w_{n_1+k}$ of the arc of the ground cycle labeled by variables of $r_k$. If we also have $q_2 \in E_N$, then with $p_2^*$ replacing $p_1$, the same definition stands for vertices $v^*_{k},w^*_k, k\in  \mcal I_2$. We have denoted by $\mcal I_1\subset [2n_1]$ and $\mcal I_2\subset [2n_2]$ the set of indices such that $(v_k,w_k)$ and $(v^*_k,w^*_k)$ respectively are defined. If $q_i$ does not belong to $E_{n}$ we set $\mcal I_i = \emptyset$.

\begin{definition} We denote by $\mcal P^\nu_{red}(V)\subset \mcal P^\nu(V)$ the set of \emph{reduced} partitions, such that $v_k$ and $w_k$ do not belong to a same block for $k\in \mcal I_1$, and the same property holds for $v^*_k$ and $w^*_k$ and $k\in \mcal I_2$.
\end{definition}

\begin{lemma}\label{Lem:Reduceness} Notations are as in Proposition \ref{Prop:SecondAsympIndep} and as in \eqref{Lesq1q2} for $q_1,q_2$ and at least one belongs to $E_n$ for some $n\geq 2$. Then we have
	$$ 	\Phi_N^{(2)}(q_1,q_2^*)  = \sum_{   \pi \in \mcal P^\nu_{red}(V) } N^{q^{(2)}(\pi)} \omega^{(2)}_X(\pi) \omega_Y(\pi).$$
\end{lemma}

\begin{proof} We denote $\mcal I= \mcal I_1 \sqcup \mcal I_2$ the formal union of the index sets and write   $\Phi_N^{(2)}(q_1,q_2^*)=\sum_{I  \subset \mcal I  }(-1)^{|I |}\Phi_N^{(2)}(p_1^{(I )},{p_2^{(I )}}^*),$
with $p_1^{(I )}$ the $\Delta$-monomial whose test graph is obtained from $T_{p_1}$ by identifying $v_k$ and $w_{k}$ for each $k\in I\cap \mcal I_1$, and $ p_2^{(I_2)}$ defined similarly.
	For any $\pi \in V$, let $\mcal I_\pi \subset \mcal I$ the set of indices $k$ for which $v_k\sim_\pi w_{k}$, in disjoint union with  indices $k$ such that $v^*_k\sim_\pi w^*_{k}$. We get $
		\Phi_N^{(2)}(q_1,q_2^*) =   \sum_{I \subset \mcal I} (-1)^{|I|} \sum_{ \substack{ \pi \in \mcal P^\nu(V)\\ \mathrm{s.t.} \   I \subset \mcal I_\pi }} N^{q^{(2)}(\pi)} \omega^{(2)}_X(\pi) \omega_Y(\pi)$ and we can exchange the two sums  $\Phi_N^{(2)}(q_1,q_2^*) =  \sum_{ \pi \in \mcal P^\nu(V)}  \Big(\sum_{  I \subset \mcal I_\pi} (-1)^{|I|} \Big)  N^{q^{(2)}(\pi)} \omega^{(2)}_X(\pi) \omega_Y(\pi).$
The sum over $I\subset \mcal I_\pi$ vanishes if $\mcal I_\pi$ is not the empty set. Hence the result.
\end{proof}

\begin{corollary}\label{Cor:No42} Let $q_1, q_2$ as in \eqref{Lesq1q2} such that at least one belongs to $E_n$ for some $n\geq 2$. Then we have
	$$\Phi_N^{(2)}(q_1,q_2^*) = \sum_{\pi \in \mcal P^\nu_{red}(V)}\one\big(T^\pi \in \boldsymbol{ \mcal{DU}} \big)\prod_{ C \in \mcal C_Y(T^\pi)} \frac 1 N \Tr^0 \big[ C(\mbf Y_N) \big] +o(1).$$
\end{corollary}

\begin{proof}We assume that $q_1$ is in $E_n$, i.e. $n_1\geq 2$ or $r_1$ is not $\Delta$-invariant.  By Proposition \ref{Prop:SecondAsympIndep}, it suffices to prove that there is no partition $\pi$ in $\mcal P^\nu_{red}(V)$ such that $T^\pi$ is of 4-2 tree type. We prove that if $\pi\in\mcal P^\nu(V)$ is such that the quotient $T_1^\pi$ of the cactus associated to $p_1$ is of double-tree type, then $\pi$ is not reduced, which implies the desired assertion.

Let $\pi\in\mcal P^\nu(V)$ such that $T_1^\pi\in \boldsymbol{ \mcal{DT}}$. Denote by $C$ the ground cycle of $T_{p_1}$. Then $C$ induces a closed path $C^\pi$ on $T_1^\pi$. We call \emph{colored component} of a test graph a maximal connected subgraph (with at least one edge) whose edges are all labeled either by a same variable $x$ in $\mbf x$, or by variables in $\mbf y$. We call \emph{graph of colored components} and denote $\mcal {GCC}(C^\pi)$ the undirect bipartite graph such that
\begin{itemize}
	\item the vertex set  is the disjoint union of the set $\mcal V_1$ of colored components of $C^\pi$, and the set $\mcal V_2$ of vertices of $C^\pi$ that belong to least two colored components,
	\item each $v \in \mcal V_2$ is connected by an edge to $S\in \mcal V_1$ whenever $v\in S$.
\end{itemize}

Since the graph of deterministic components $\mcal {GDC}(T^\pi)$ is of double-tree type, so is $\mcal {GDC}(C^\pi)$, and hence the graph of colored components $\mcal {GDC}(C^\pi)$ is a tree. Since the ground cycle of $p_1$ has at least two arcs labeled by different variables, $\mcal {GDC}(C^\pi)$ has at least two vertices in $\mcal V_1$. Hence the tree $\mcal {GDC}(C^\pi)$ has at least two leaves. A leaf of this graph is necessarily in $\mcal V_1$ since a vertex of $\mcal V_2$ is connected to at least two elements of $\mcal V_1$ by definition.

Let $e$ be an edge of the ground cycle of $T_1$ such that the corresponding step of $C^\pi$ belong to a leaf of $\mcal {GCC}(C^\pi)$. This edge belongs to an arc associated either to some $s_k$ or to a non $\Delta$-invariant $r_k$. We assume this arc associated to $s_k$, there is no modification of the reasoning in the other case. All edges of $s_k$ belong to the same colored component of $C$, so in the quotient $C^\pi$ they also belong to the same colored component. The extremities of the arc $s_k$ also belong to another colored component. Since we have considered a leaf, there is only one vertex of the component of $s_k$ with this property, hence $v_k$ and $w_k$ are equal. As a conclusion, $\pi$ is not reduced. 
\end{proof}

We a small abuse, we call \emph{double cycle} of $T^\pi \in  \boldsymbol{\mcal{DU}}$ the maximal subgraph that covers the cycle of $\overline{\mcal {GDC}(T^\pi)}$ when multiplicity is forgotten.

\begin{corollary}\label{PreOrtho} Assume $q_1\in E_n$ for some $n\geq 2$. Then for any $\pi\in \mcal P^\nu_{red}(V)$ such that $T^\pi \in \boldsymbol{\mcal{DU}}$, each arc associated to $s_k$ or to a non  $\Delta$-invariant $r_k$ has at least one edge on the double cycle of $T^\pi$. \end{corollary}

\begin{proof}We consider $ \pi \in \mcal P^\nu(V)$ such that $T^\pi \in  \boldsymbol{\mcal{DU}}$. We repeat the same reasoning as for the previous corollary. The graph of colored component $\mcal {GCC}(C^\pi)$ is a unicyclic graph. If it has a leaf, then corresponds an identification and so $\pi$ is not reduced. 
\end{proof}

\begin{corollary}\label{Ortho} The spaces $E_n,n\geq 0$ are asymptotically orthogonal for $\Phi^{(2)}_N$, that is: $\Phi^{(2)}_N(q,{q'}^*)\limN 0$, for any $q\in E_n$ and ${q'}^*\in E_{n'}$ such that $n\neq n'$. 
\end{corollary}

\begin{proof} Each element $q \in E_n$ for some $n\geq 2$ is linear combination of element of the form $q=q_1 $ as in the beginning of the section, given by reducing terms of a cyclic alternating product of $\Delta$-monomials. As well for each $\Delta$-monomial $q \in E_1$, either its has zero full degree in $\mbf x$ and so $\Phi^{(2)}_N(q,\, \cdot \, ) =0$, either by the change block method $\Phi^{(2)}_N(q,\, \cdot \, ) = \Phi^{(2)}_N(\tilde q,\, \cdot \, )$ where $\tilde q$ is of the form $\tilde q = p_1 = s_1r_1\in \Dxy\langle x \rangle$ as above (in the case where $n_1=1$ and $r_1$ is $\Delta$-invariant). Similarly, we can assume that $q'$ is of the form $p_2$ as above.

Hence it is sufficient to prove with these notations that $\Phi_N^{(2)}(q_1,q_2^*)$ tends to zero if $q_1 \in \mcal E_n$ and $q_2 \notin E_n$. But if $T^\pi$ is a double-unicyclic type graph, then the sequence of labels on the unique cycle of $\mcal {GCC}(T_1^\pi)$ and the same sequence for $\mcal {GCC}(T_2^\pi)$ with reverse order must coincide. Hence by Corollary \ref{PreOrtho}, there is a $\pi\in \mcal P^\nu_{red}(V)$ such that $T^\pi \in \boldsymbol{\mcal {DU}}$ only if $q_2 \in E_n$. 
\end{proof}

\subsection{Proof of (ii)}\label{Sec:ProofMainTh4}
In this section we  assume $q_1,q_2\in E_n$ for some $n\geq 2$. 
In order to compare $\Phi^{(2)}$ with the first-order $\Delta$-distribution $\Phi$, we also set $m_0=m_1m_2^*$, $p_0=p_1p_2^*$, and $q_0  = q_1q_2^*$.
We denote by $T_{0}$ the w.o. simple cycle associated to $\Delta(m_0)$, by $V_0$ its vertex set, and by $\nu_0$ the partition of $V_0$ such that $T_{p_0}=T_{0}^{\nu_0}$ is the w.o. cactus associated to $p_0$. We consider the vertices $v_{k}, w_{k}$ and $v^*_{k}, w^*_{k}$ in the ground cycle of the graph $T_{0}$ with same definition as before. By \cite[Lemma 4.5]{ACDGM}, we have the analogue of \eqref{Lem:Reduceness}, namely 		\eqa\label{Eq:Reduceness}
			\Phi_N(q_0) & = & \sum_{  \pi \in \mcal P^{\nu_0}_{red}(V_0) }   N^{q^{(1)}(\pi)} \omega^{(1)}_X(\pi) \omega_Y(\pi)\\
			& =& \sum_{  \pi \in \mcal P^{\nu_0}_{red}(V_0) } \one\big( T_0^\pi \in \boldsymbol{ \mcal{DT}}\big) \prod_{ C \in \mcal C_Y(T_0^\pi)} \frac 1 N \Tr^0 \big[ C(\mbf Y_N) \big] +o(1).
	\qea
Note that the ground cycle of $T_0$ is made of arcs from the $s_k$ and the non $\Delta$-invariant $r_k$ of both $\Delta$-monomials $p_1$ and $p_2^*$. The arc of $s_1$ is neighbor of the arc of $s_{n+1}$ and they can be labeled by a same Wigner matrix. If $r_{n}$ and $r_{2n}$ are not  $\Delta$-invariant, they also form two neighbouring arcs in the ground cycle of $T_0$, if they are  $\Delta$-invariant then $s_{n}$ and $s_{2n}$ have this property. Otherwise, all the other arcs have neighbors labeled by a variable in a different family. 

Assume there is a partition $\pi \in \mcal P^{\nu_0}_{red}(V_0)$ such that $T_0^\pi$ is of double-tree type. Recall from the proof of Corollary \ref{Cor:No42} the definition of the graph of colored components $\mcal {GCC}(T_0^\pi)$. This graph is a tree with at least two leaves, and a leaf cannot contain an edge of an arc whose extremities are not identified if its neighbors are labeled by different families. Hence these leaves corresponds to the two possible cases of adjacent arcs with same labels family. Necessarily $s_1$ and $s_{n+1}$ are labeled by a same variable and $w_1\sim_\pi v_1^*$. Moreover, $r_{n}$ is  $\Delta$-invariant if and only if $r_{2n}$ is $\Delta$-invariant, in which case $s_{n}$ and $s_{2n}$ are labeled by a same variable and $v_k\sim_\pi w^*_k$. If $ r_{n}$ and $r_{2n}$ are not $\Delta$-invariant, then $v_{2n}\sim_\pi w_{2n}^*$.

Reasoning by induction on the subgraph obtained by removing the leaves,  $\Phi_N(q_0) $ does not converges to zero only if for all $k\in [n]$, $s_k$ and $s_{n+k}$ have same label, $r_k$ is  $\Delta$-invariant if and only if $r_{n+k}$ is  $\Delta$-invariant. In this case we get
	\eqa\label{Eq:KeyIdentifications}
		\Phi_N(q_0) = \sum_{ \substack{  \pi \in \mcal P^{\nu_0}_{red}(V_0) \\ v_k\sim_\pi w_k^* \ \forall k>1} } \one\big( T_0^\pi \in \boldsymbol{ \mcal{DT}}\big) \prod_{ C \in \mcal C_Y(T_0^\pi)} \frac 1 N \Tr^0 \big[ C(\mbf Y_N) \big].
	\qea

We are now in position to prove Mingo-Speicher formula for elements $q_1,q_2\in E_n,n\geq $  as in \eqref{Lesq1q2}, that $\Phi^{(2)}\big( q_1,q_2^*) = \sum_{i=0}^{n-1} \Phi\big( q_1(q_2)_{[i]}^*\big)$ where 
	$$(q_2)_{[i]}  =  \tilde s_{n_1+1+i}\tilde  r_{n_1+1+i}  \cdots \tilde s_{n_1+n_2} \tilde r_{n_1+n_2}\tilde s_{n_1+1} \tilde r_{n_1+1} \cdots \tilde s_{n_1+i} \tilde r_{n_1+i}.$$
The expression computed for $\Phi\big( q_1(q_2)^*\big)$ extends obviously fo $\Phi\big( q_1(q_2)_{[i]} ^*\big)$, so we have all combinatorial expressions to demonstrate the equality. Yet the exercise of writing a bijection for this proof is delicate, especially when trying to emphasis the crucial role of assumption (H5). We propose instead to prove the formula thanks to a variation of Schwinger-Dyson equation for $\Phi$ and $\Phi^{(2)}$, following the strategy of the first order convergence. 

\begin{lemma}\label{FirstOrderSD2} With notations as in the beginning of the section, writing $s_1=xs'_1$, we have the following recurrence
	\eq
		\Phi(q_0)		& = & \sum_{s_1 = x\ell x r x } \Phi(\ell) \Phi\Big[ \big( rx - \Delta(rx)\big)  \tilde r_1 \tilde s_2 \cdots \tilde s_{n} \tilde r_{n} \times q_2^* \Big]\\
		& & \ \ + \sum_{s_{n+1} = \ell x r} \Phi(r) \Phi\Big[ s'_1  \tilde r_1 \tilde s_2 \cdots \tilde s_{n} \tilde r_{n} \times  \tilde r_{2n} \tilde s_{2n}  \cdots \tilde r_{n+1} \ell \Big].
	\qe
\end{lemma}

\begin{proof} By \eqref{Eq:Reduceness} and Proposition \ref{Prop:AsympIndep}, 
	$$\Phi(q_0) = \Nlim  \sum_{\pi  \in \mcal P^\nu_{red}(V_0)}  \one\big( T_0^\pi \in \boldsymbol{ \mcal{DT}}\big) \underbrace{\prod_{ C \in \mcal C_Y(T_0^\pi)} \frac 1 N \Tr^0 \big[ C(\mbf Y_N) \big]}_{=:\tilde \omega_Y(\pi)}.$$
Moreover, by \eqref{Eq:KeyIdentifications}, we can restrict to partitions such that $v_2 \sim_\pi w_2^*$ (or equivalently $w_1 \sim_\pi v_1^*$). Let $e$ denotes the edge of $T_{p_0}$ associated to the first letter $x$ of the word $p_0$ in $T_0$. For any $\pi$ as in the above sum and such that $T_0^\pi \in \boldsymbol{ \mcal{DT}}$, let $e'$ be the twin edge of $e$. Then $e'$ must belong to the same simple cycle as $e$ in the graph obtained from $T_0$ by identifying $w_1$ and $v_1^*$. We writing in short $e' \in s_k$ when $e'$ belongs to the arc defined by $s_k$, we then get $e' \in s_1$or $e'\in s_{n+1}$.    

If $s_1=xs'_1\neq x$, then $s'_1$ is not the empty word and it is of the form $s'_1=s''_1x$. Assume $e' \in s_1$.  Because of the condition $v_1\not\sim_\pi w_1$, $e'$ cannot be the edge corresponding to the last letter $x$ of the word $s'_1$. The possible choices of an edge $e' \in s_1$ corresponds then to the choices of a decomposition $s'_1 = \ell x r x$. The graph $T_0^\pi$ is a quotient of the graph $T_{0,e\sim e'}$ obtained by identifying source of $e$ and target of $e'$ and reciprocally. Removing $e,e'$ from $T_{0,e\sim e'}$ yields two connected components $T_{\ell}$ and $T_{p'}$ where $p'=  rx r_1 s_2 r_2  \cdots s_{n} r_{n}\times  r_{2n}s_{2n}\cdots r_{n+1} s_{n+1}$. Let $\pi_1$ and $\pi_2$ be the restrictions of $\pi$ on the vertex sets of these two subgraphs respectively. Then the quotients $T_{ \ell}^\pi$ and $T_{p'}^\pi$ are also of double tree type. The partition $\pi_2$ is subject to the constraint that extremities of the arcs of the ground cycle of $p'$ must not be identified, we denote in short $\pi_2\in \mcal P_{red}(V_{p'})$. The partition of $\pi_1$ is subject to no constraint. Hence, by Proposition \ref{Prop:AsympIndep} and the same formula as \eqref{Eq:Reduceness} for $p'$, 
\eq
	\lefteqn{\sum_{ \substack{ \pi \in \mcal P^\nu_{red}(V_0) , \ \mathrm{ s.t.}  \ e' \in s_1}} \one\big( T_0^\pi \in \boldsymbol{\mcal {DT}} \big)\tilde \omega_Y(\pi) }\\
	&= & \sum_{ s'_1 = \ell x r x}   \sum_{ \pi_1\in \mcal P(V_\ell)} \one\big( T_\ell^\pi \in \boldsymbol{\mcal {DT}} \big) \tilde \omega_Y(\pi_1)    \sum_{   \pi_2 \in \mcal P_{red}(V_{p'})} \one\big( T_{p'}^{\pi_2} \in \boldsymbol{\mcal {DT}} \big) \tilde \omega_Y(\pi_2) \\
	& \limN&  \sum_{ s_1 = x\ell x r x} \Phi[\ell] \Phi\Big[ \big( rx - \Delta(rx)\big)  \tilde r_1 \tilde s_2 \cdots \tilde s_{n} \tilde r_{n} \times \tilde q_2^* \Big].
\qe

Assume now that $e' \in s_{n+1}$, which corresponds to a decomposition $s_{n+1} = \ell x r$. As for the case $e'\in s_1$, $T_0^\pi$ is then a quotient of $T_{0,e\sim e'}$. Removing $e$ and $e'$ from this graph yields two connected components $T_{r}$ and $T_{p'}$ where $p' = s'_1 r_2 s_2 \cdots s_{n} r_n s_{2n} r_{2n} \cdots r_{n+1} \ell$. The restriction of $\pi$ on the vertex set of $T_r$ has no constraints, the restriction on the vertex set of $T_{p'}$ must not identify extremities of all arcs of the ground cycle but the first and last arcs that correspond to $s'_1$ and $\ell$ (conditions $v_1\not\sim_\pi w_1$ and $v_{n+1}\not\sim_\pi w_{n+1}$ are always satisfied when $e'\in s_{n+1}$ since $w_1$ and $v_{n+1}$ are vertices of $T_\ell$ whereas $v_1$ and $w_{n+1}$ are vertices of $T_{p'}$). Hence the claimed formula.\end{proof}

\begin{lemma}\label{SecondOrderSD2} With above notations, we write $s_k=\big(xs'_k - \Delta(xs'_k)\big)$. Then
	\eq
		\lefteqn{\Phi^{(2)}\big( q_1,q_2)}\\
		& = & \sum_{s_1 = x\ell x r x } \Phi(\ell) \Phi^{(2)}\Big[ \big( rx - \Delta(rx)\big)  \tilde r_1 \tilde s_2 \cdots \tilde s_{n} \tilde r_{n},  \tilde q_2^*    \Big]\\
		& & \ \ + \sum_{i=1}^{n}\  \sum_{s'_{n+i} = \ell x r}  \Phi\Big[ s'_1  \tilde r_1 \tilde s_2 \cdots \tilde s_{n} \tilde r_{n} \Delta(r) \tilde r_{n+i-1}\tilde s_{n+i-1} \cdots r_{n+1}s_{n+1} r_{2n} s_{2n} \cdots \tilde r_{n+i} \ell \Big].
	\qe
\end{lemma}

\begin{proof} Let $\pi$ be a partition such that $\pi\geq \nu$, $T^\pi$ is of double-unicyclic type, and $v_k\not\sim_\pi w_k$ for any $k$. Let $e$ be the edge of $T$ associated to the first letter $x$ of $s_1=xs'_1$ and let $e'$ be the edge twined with $e$ by $\pi$. 

Condition $v_1\not\sim w_1$ implies as before that $e'$ is not associated to the last letter of $s_1$. Assume that $e'\in s_1$, which hence corresponds to a decomposition $s_1 = \ell x rx$. The double edge is not in the double cycle of $T^\pi$, and so the graph obtained by removing $e$ and $e'$ from $T^\pi$ has two connected components: one is a double tree quotient of $T_\ell$, the other is a double unicyclic graph quotient of the graph obtained from $T$ by replacing $xm_1$ by $rx$ (Lemma \ref{PreOrtho} ensures that the graph of double-unicyclic type is not the quotient of $T_\ell$). The same reasoning as for Lemma \ref{FirstOrderSD2} hence gives  the first term.

The usual argument on the graph of colored components implies that $e'\not\in s_k$ for $2\leq k\leq n$. Assume hence that $e' \in s_k$ for $k>n$. The choice of $k$ correspond to the choice of $i\in [1\etc n]$, and the choice of the edge $e'$ in the arc of $s_k$ corresponds to the choice of a decomposition $s_{n+i} = \ell x r$. Moreover, removing $e$ and $e'$ from $T^\pi$ produces a graph of double-tree type. The argument of the beginning of this section implies that identification $w_k\sim_\pi v_k^*$ must be satisfied. By \eqref{FirstOrderSD2}, for $i=1$
\eq
	\lefteqn{\sum_{ \substack{ \pi \in \mcal P^\nu_{red}(V) , \\ \mathrm{ s.t.}  \ e' \in s_{n+1}}} \one\big( T_0^\pi \in \boldsymbol{\mcal {DT}} \big)\tilde \omega_Y(\pi) }\\
	&= & \sum_{s'_{n+1} = \ell x r}  \Phi\Big[ s'_1  \tilde r_1 \tilde s_2 \cdots \tilde s_{n} \tilde r_{n} \times r \times   \tilde r_{n+i-1}\tilde s_{n+i-1} \cdots r_{n+1}s_{n+1} r_{2n} s_{2n} \cdots \tilde r_{n+i} \ell \Big]
\qe

We claim that the source and target vertices of the path corresponding to $\ell$ must be identifying. Indeed, the labels of the edges are $x$, so they cannot be twined with the edges of $s_{n}$, whose edges must appear in the double cycle just before those of $s_1$ by Corollary \ref{PreOrtho}. Hence in above formula we can replace $r$ by $\Delta(r)$. We get the analogue for formula for $i>1$ by shifting the indices in $q_2$, and hence the second contribution in the lemma.
\end{proof}

By Assumption (H5), we have 
	\eq
	\lefteqn{\Phi\Big[ s'_1  \tilde r_1 \tilde s_2 \cdots \tilde s_{n} \tilde r_{n} \Delta(r) \tilde r_{n+i-1}\tilde s_{n+i-1} \cdots r_{n+1}s_{n+1} r_{2n} s_{2n} \cdots \tilde r_{n+i}  \Big]}\\
	&  = & \Phi(r)\Phi\Big[ s'_1  \tilde r_1 \tilde s_2 \cdots \tilde s_{n} \tilde r_{n} \tilde r_{n+i-1}\tilde s_{n+i-1} \cdots r_{n+1}s_{n+1} r_{2n} s_{2n} \cdots \tilde r_{n+i}  \Big].
	\qe
Assume first that $s_1$ has ground degree 1 or 2. Then both in Lemmas \ref{FirstOrderSD2} and \ref{SecondOrderSD2} there are no decomposition $s'_1 = \ell x r x$ and the first sum is zero. We hence get that
$\Phi^{(2)}\big( q_1,q_2^*) = \sum_{i=0}^{n-1} \Phi\big( q_1(q_2)_{[i]}^*\big)$
as claimed.  Assume now that  $s_1$ has ground larger than 2. Subtracting the formulas 
	$$\Phi^{(2)}[q_1,q_2] - n   \Phi[q_1\mcal M(q_2)^*] = \sum_{s_1 = x \ell x r x} \Phi(\ell)\big( \Phi^{(2)}[q'_1,q_2^*] - n \Phi[q'_1\mcal M(q_2)^*] \big),$$
	where $q'_1 = rx - \Delta(rx)$. By induction of the ground degree of $s_1$, we hence get that the formula is always valid.

\subsection{Proof of (iii)}\label{Sec:ProofMainTh3}

By linearity, traciality and  invariance under $\Delta$ it is sufficient to prove that for any non-mixing $\Delta$-monomial $p$,
	\eqa\label{Eq:SecondOrderTransitivity}
		\Phi_N^{(2)}(p, \, \cdot \, ) = \sum_{B \in \sigma(p)}  \Phi_N^{(2)}(p_B, \, \cdot \, ) \prod_{ B' \neq B} \Phi_N(m_{B'})+o(1).
	\qea

If $p \in \mbb C\langle \mbf y \rangle_\Delta$, the equality is obvious, both sides vanish. Let $p\in \mbb C\langle \mbf x, \mbf y \rangle_\Delta$ of positive full degree on a variable $x \in \mbf x$. Assume that \eqref{Eq:SecondOrderTransitivity} holds for all $\Delta$-monomials of smaller full degree. By ground block change, which preserves both side of \eqref{Eq:SecondOrderTransitivity}, we can assume that $p$ is actually of the form $xm$. We hence consider $\Delta$-monomials of the form $p=p_1=xm$ and $p_2=m'$ as in the previous sections. In the graph $T_{p_1}$, let $e$ be the edge associated to the first letter $p=xm$. 
\begin{itemize}
	\item Denote by $\mcal P_1 \subset \mcal P(V)$ the set of partitions $\pi$ such that $T^\pi \in \boldsymbol{\mcal {DU}}\cup \boldsymbol{\mcal {FT}}$, $e$ is twinned with an edge $e'$ of $m$ (so $e$ is not in the double cycle if $T^\pi \in \boldsymbol{\mcal {DU}}$) and is not in the group of edges of multi. 4 if $T^\pi \in \boldsymbol{\mcal {FT}}$. 
	\item We denote by $\mcal P_2\subset \mcal P(V)$ the set of partitions $\pi$ such that $T^\pi \in \boldsymbol{\mcal {FT}}$ and $e$ belongs to the group of edges of multip. 4. 
	\item Let $\mcal P_{k}\subset \mcal P(V)$ be the set of partitions $\pi$ such that $T^\pi \in \boldsymbol{\mcal {DU}}$, $e$ belongs to the double cycle of $T_\pi$ and the double cycle of $T_\pi$ is of length $k\geq 3$.
\end{itemize}
Setting  $\alpha_k(p_1,p_2)  :=  \sum_{ \pi \in \mcal P_k} \omega_X^{(2)}(\pi)\omega_Y(\pi)$ for any $k\geq 1$ (the sum is finite), by Proposition \ref{Prop:SecondAsympIndep} we have $\Phi^{(2)}(p_1,p_2) =  \sum_{k\geq 1}\alpha_k(p_1,p_2)+o(1)$. The same reasoning as for SD equation yields the following relations. Firstly
	\eq
		\alpha_1(p_1,p_2) & = & \sum_{ p_1 = x\ell x r} \big( \Phi_N^{(2)}(\ell,m') \Phi_N(r)+\Phi_N(\ell) \Phi_N^{(2)}(r,m')\big)+o(1),
	\qe
since removing $e,e'$ form $T^\pi_{p_1}$ gives two connected components, one of double tree type, and the other will contribute with $T^\pi_{p_2}$ to form an element of $\boldsymbol{\mcal {DU}}\cup \boldsymbol{\mcal {FT}}$. By the Schwinger-Dyson equation, \eqref{Eq:TransitivityShort} and the induction hypothesis, denoting by $GB$ the ground block of $p=p_1$, we verify easily that $\alpha_1(p,\cdot \, )$ is equal to
	\eqa
		\sum_{\substack{B \in \sigma(p)\\ B \neq GB}}  \Phi_N^{(2)}(p_B, \, \cdot \, ) \prod_{B' \neq B}\Phi_N(p_{B'})+ \alpha_1(p_{GB},\cdot \, )\prod_{B \in \sigma(p)} \Phi_N(p_{B})+o(1).\label{Eq:SOT1}
	\qea

Similarly, setting $m_4(x) = \esp[|x_{12}|^4]$ when $x$ is associated to $X_N=\big( \frac{x_{ij}}{\sqrt  N}\big)$, 
	\eq
		\alpha_2(p_1,p_2) & = & \sum_{  p_1=x\ell x r }\sum_{ \substack{ B\in \sigma(p_2) \\ {p_2}_{B} = \ell' x m''xr'}} (m_4(x) - 1) \Big( \Phi_N\big( \Delta(\ell) \Delta(m') \big)  \Phi_N\big( \Delta(r) \Delta(r'\ell') \big)\\
		& & \ \ \ + \Phi_N\big( \Delta(\ell) \Delta(r'\ell') \big)  \Phi_N\big( \Delta(r) \Delta(m') \big) \Big)  +o(1),
	\qe
Indeed the choice of a block $B$ and of a decomposition ${p_2}_{B} = \ell' x m''xr'$ correspond to the choice of edges $e'',e'''$ such that $\{e,e',e'',e'''\}$ form the edge of multip. 4. The graph obtained by removing the edge of multip. 4 is then a union of double-tree type graphs. By \eqref{Eq:TransitivityShort}, we can split the terms $\Phi_N\big( \Delta(\ell) \Delta(m') \big)= \Phi_N(\ell)\Phi_N(m')+o(1)$. This allows to prove, with the same arguments as for \eqref{Eq:SOT1}
	\eqa
		\alpha_2(p,\cdot \, )  & = &  \alpha_2(p_{GB},\cdot \, )\prod_{\substack{ B \in \sigma(p)\\ B\neq GB}} \Phi_N(p_{B})+o(1).\label{Eq:SOT2}
	\qea
Finally, for any $k\geq 3$, we consider all decompositions $p_1 = xm_1 x m_2 x \cdots x m_k,$ and ${p_2}_B = \ell x m'_1  \cdots x m'_{k-1}x r,$ for any block $B\in \sigma(p_2)$ in the formula below
	\eq
		\alpha_k(p_1,p_2) & =&  \sum_{ \substack{ p_1 = \cdots \\ {p_2}_B=\cdots }} \sum_{i=1}^k \Phi( \Delta(m_1) \Delta(m'_{i+1})) \cdots \Phi( \Delta(m_k)\Delta( m'_{i+k})) +o(1),
	\qe
with indices modulo $k$ and $m'_{k}:=r\ell$. The decompositions of $p_1$ and $p_2$ represent the choice of the edges of the double cycles. The sum over $i$ represent the $n$ different way we can merge the $n$-cycles of $T^\pi_1$ and $T^\pi_2$ in a double cycle. As in the previous case, this allows to obtain that \eqref{Eq:SOT2} is valid for $\alpha_k$ replacing $\alpha_1$ in both sides. Together with \eqref{Eq:SOT1}, these relations imply  Formula \eqref{Eq:SecondOrderTransitivity}.

\begin{remark} The Leibniz property is stated only on $E_1$. It is expected to hold for all $\Delta$-polynomials, the restriction has only purpose to simplify the proof.
\end{remark}

\end{document}